\documentclass{amsart}

  \newcommand\ma[1]{{\color{black} #1}}
 \newcommand\ibt[1]{{\color{black} #1}}
 
\usepackage{floatrow}
\usepackage{caption}
 \usepackage[foot]{amsaddr}
\let\oldtocsection=\tocsection
 
\let\oldtocsubsection=\tocsubsection 
 
\let\oldtocsubsubsection=\tocsubsubsection

\renewcommand{\tocsection}[2]{\vspace{0.5em}\hspace{0em}\oldtocsection{#1}{#2}}
\renewcommand{\tocsubsection}[2]{\vspace{0.5em}\hspace{1em}\oldtocsubsection{#1}{#2}}
\renewcommand{\tocsubsubsection}[2]{\vspace{0.5em}\hspace{2em}\oldtocsubsubsection{#1}{#2}}
\usepackage{graphicx,cancel,xcolor,hyperref,comment,graphicx,geometry}
 
\usepackage{amsopn}

\DeclareMathOperator{\supp}{supp}

\usepackage{tikz}
\usepackage{graphicx,url,etoolbox}
\usepackage{lipsum}
\usepackage{dsfont}
 \usepackage{amssymb}

\usepackage{graphicx,cancel,xcolor,hyperref,comment,graphicx,geometry}

\usepackage{tikz}
\usetikzlibrary{decorations.pathreplacing}
\usepackage{graphicx,url,etoolbox}
\usepackage{lipsum}
\usepackage{dsfont}

\usepackage{fancyhdr}
\usepackage{lastpage}

\newtheorem{theoreme}{Theorem}

\newtheorem{rem}[theoreme]{Remark}
\theoremstyle{definition}

\usepackage{graphicx}
\usepackage{subcaption}

\numberwithin{equation}{section}

 \renewenvironment{proof}{{\bfseries \noindent Proof.}}{\demo}
\newcommand\xqed[1]{%
  \leavevmode\unskip\penalty9999 \hbox{}\nobreak\hfill
  \quad\hbox{#1}}
\newcommand\demo{\xqed{$\square$}}

\hypersetup{bookmarks, colorlinks, urlcolor=blue, citecolor=blue, linkcolor=blue, hyperfigures, pagebackref,
    pdfcreator=LaTeX, breaklinks=true, pdfpagelayout=SinglePage, bookmarksopen=true,bookmarksopenlevel=2}

\newtheorem*{nonamethm}{\nonamethmname}
\newcommand{\nonamethmname}{}

\newenvironment{genthm*}[1]
 {\renewcommand{\nonamethmname}{#1}\nonamethmcheck}
 {\endnonamethm}

\newcommand\nonamethmcheck[1][]{%
  \if\relax\detokenize{#1}\relax
    \nonamethm\relax
  \else
    \nonamethm[#1]%
  \fi
  \mbox{}%
}

\usepackage{comment}

\def\R{\mathbb R}

\def\la {{\lambda}}

\newcommand {\nc}   {\newcommand}
\nc {\be}   {\begin{equation}} \nc {\ee}   {\end{equation}} \nc
{\beq}  {\begin{eqnarray}} \nc {\eeq}  {\end{eqnarray}} \nc {\beqs}
{\begin{eqnarray*}} \nc {\eeqs} {\end{eqnarray*}}
\def\edc{\end{document}}

\providecommand{\abs}[1]{\lvert#1\rvert}
   
\usepackage{tikz}
\usetikzlibrary{decorations.pathmorphing,patterns,scopes,intersections,calc}
\usepackage{caption}
\usepackage{float}
\usepackage{soul}

\newcounter{dummy} 
\numberwithin{dummy}{section}
\newtheorem{Theorem}[dummy]{Theorem}

\newtheorem{Lemma}[dummy]{Lemma}

\newtheorem{Proposition}[dummy]{Proposition}
\newtheorem{Remark}[dummy]{Remark}

\newtheorem{Hypothesis}[dummy]{Hypothesis}
\numberwithin{equation}{section}

\newtheoremstyle{break}
  {\topsep}{\topsep}%
   {\itshape}{}%
  {\bfseries}{}%
  {\newline}{}%
\theoremstyle{break}
\newtheorem{ex}[dummy]{Example}

\begin{document}
\title[\fontsize{7}{9}\selectfont  ]{Stability for degenerate wave equations with drift under simultaneous degenerate damping}
\author{Mohammad Akil$^1$ , Genni Fragnelli$^2$ , Ibtissam Issa$^3$ }

\address{$^1$ Universit\'e Polytechnique Hauts-de-France, C\'ERAMATHS/DEMAV, Le Mont Houy 59313 Valenciennes Cedex 9-France.}

\address{$^2$ Department of Ecology and Biology, Tuscia University, Largo dell’Universit\'a, 01100 Viterbo - Italy.}
\address{$^3$ Universit\'a  degli studi di Bari Aldo Moro-Italy, Dipartimento di Matematica, Via E. Orabona 4, 70125 Bari - Italy.}

\email{mohammad.akil@uphf.fr}
\email{genni.fragnelli@unitus.it}
\email{ibtissam.issa@uniba.it}

\keywords{Degenerate wave equation; degenerate damping, drift, exponential stability.}

\begin{abstract}
In this paper we study the stability of two different problems.
The first one is a one-dimensional degenerate wave equation with degenerate damping, incorporating a drift term and a leading operator in non-divergence form. In the second problem we consider a system that couples degenerate and non-degenerate wave equations, connected through transmission, and subject to a single dissipation law at the boundary of the non-degenerate equation. In both scenarios, we derive exponential stability results.
\end{abstract}

\maketitle
\pagenumbering{roman}
\maketitle
\tableofcontents
\pagenumbering{arabic}
\setcounter{page}{1}
\section{Introduction}

Degenerate partial differential equations (PDEs) are a subclass of PDEs in which some of the coefficients or terms lose their independence or become dependant on the same variable. As a result, the PDE could lose its typical characteristics, including ellipticity or hyperbolicity, which can have a substantial impact on how solutions behave.

In the context of wave equations, one intriguing aspect of wave equations arises when certain coefficients or terms become dependent on the same variable, leading to the emergence of degenerate wave equations. In such scenarios, the usual well-behaved properties of wave equations, such as hyperbolicity and ellipticity, may no longer hold, resulting in unique and sometimes counterintuitive wave behaviors. In some physical systems, losing ellipticity can cause novel phenomena, such as the appearance of many solutions with the same energy (degenerate energy levels) or the non-uniqueness of solutions. In materials with anisotropic properties, the wave equation can become degenerate, meaning that certain directions of wave propagation exhibit different behaviors compared to others. For example, in certain crystals, the speed of wave propagation may vary depending on the direction of the wave, leading to degeneracy in the wave equation. Also, in fluid dynamics,  in the study of shallow water waves, some specific flow conditions can cause the wave equation to become degenerate.

In many real-world contexts, including camouflage (the fabrication of devices that render their operators invisible to outside observation) \cite{Cloaking},  L\'evy noise \cite{BiswasMajee}, meteorology \cite{BADII1999713},  and biological population  \cite{NIKAN}, degenerate partial differential equations (PDEs) give rise to control and inverse problems. Intricate mathematical problems associated with degenerate PDEs have been brought to light by these many applications. For example, the field of semiconductor physics and device engineering provides a significant physical application where degenerate PDEs play a crucial role in understanding and optimizing the behavior of modern electronic devices.\\
These circumstances often include operators that are not uniformly elliptic because their diffusion coefficients vary spatially. However, these operators become uniformly elliptic in small areas of the spatial domain that are at positive distance from the degenerate area. Degeneracy can occur on either the boundary or an internal submanifold.\\
Recently, there has been a growing interest in the study of degenerate wave equations.  To the best of our knowledge, there has been no investigation conducted to date that addresses the case in which both the wave and the damping are simultaneously considered degenerate. Furthermore, the concept of internally localized degenerate damping has not been explored in the existing literature. Thus,  in this paper, the main focus is twofold. Firstly, it aims to explore the stability of a degenerate wave equation that incorporates locally degenerate damping.  Significantly, this study distinguishes itself as the first to simultaneously consider the degeneracy of both the wave and the damping in the equation. Also, in this paper we get a better generalized condition for the exponential stability than that in \cite{fragnelli2022linear} in Hypothesis 4. Thus, the system for this case is as follows:
\begin{equation}\label{S1}
\left\{\begin{array}{lll}
u_{tt}-a(x)u_{xx}-b(x)u_{x}+h(x)u_{t}=0,& (x,t) \in (0,1)\times\R^+_{\ast},\\[0.1in]
u(0, t)=u(1,t)=0, &t\in \R^+_{\ast},\\[0.1in]
u(x,0)=u_0(x), u_t(x,0)=u_1(x),&x\in (0,1)
\end{array}\right.
\end{equation}
where $a, b\in C ^{0}[0,1]$, with $a>0$ on $(0,1]$, $a(0)=0$ and $\dfrac{b}{a}\in L^1(0,1)$. Hence, if $a(x)=x^K$, $K>0$, we can consider $b(x)=x^m$ for any $m>K-1$. The damping coefficient function $h(\cdot):(0,1)\rightarrow \R^{+}\cup \{0\}$, belongs to $L^{\infty}(0,1)$ and satisfies:
$h(x)=0$ for $x\in (0,x_1)\cup(x_2,1)$ (assuming $0<x_1<x_2<1$) and $h(x)=h_1(x)=\abs{(x-x_1)^{\alpha_1}(x-x_2)^{\alpha_2}}$ for $x\in [x_1,x_2]$ with $\alpha_1, \alpha_2\geq 0$ . The initial data $u_0$ and $u_1$ belong to suitable weighted spaces.
The degeneracy of \eqref{S1} at $x = 0$ is measured by the parameter $K$ defined by
\begin{equation}\label{Degeneracy}
K:=\sup _{x\in (0,1]}\dfrac{x\abs{a^{\prime}(x)}}{a(x)}.
\end{equation}
We say that $a$ is weakly degenerate at $0$, (WD) for short, if 
\begin{equation}\label{WD}\tag{WD}
a\in C^0[0,1]\cap C^1(0,1] \quad  \text{and} \quad K\in (0,1)
\end{equation}
and we say that $a$ is strongly degenerate at $0$, (SD) for short, if 
\begin{equation}\label{SD}\tag{SD}
a\in C^1[0,1] \quad \text{and}\quad K\in [1,2).
\end{equation}
Here we assume $K<2$ because it is essential in the calculation that will be conducted below. 

Secondly, it aims to study the coupling of a degenerate and a non degenerate wave equations via transmission with only one dissipation law acting at the end of the non degenerate part.  The system is given as follows
\begin{equation}\label{System2}
\left\{\begin{array}{lll}
u^1_{tt}-a(x)u^1_{xx}-b(x)u^1_{x}=0,& (x,t) \in (0,1)\times\R^+_{\ast},\\[0.1in]
y^1_{tt}-y^1_{xx}=0, &(x,t) \in (1,L)\times\R^+_{\ast},\\[0.1in]
u^1(0,t)=0, u^1(1,t)=y^1(1,t), &t\in \R^+_{\ast},\\[0.1in]
(\eta u^1_x)(1,t)=y^1_x(1,t), y^1_x(L,t)=-y^1_t(L,t),&t\in \R^+_{\ast},\\[0.1in]
u^1(x,0)=u^1_0(x), u^1_t(x,0)=u^1_1(x),&x\in (0,1),\\[0.1in]
y^1(x,0)=y^1_0(x), y^1_t(x,0)=y^1_1(x), &x\in (1,L),
\end{array}\right.
\end{equation}
where $L>1$, $a$, $b$ are defined as in system \eqref{S1},  and $\eta$ is the well-known absolutely continuous weight function
$$\eta(x):=\exp\left\lbrace\int _{\frac{1}{2}}^{x}\dfrac{b(s)}{a(s)}ds\right\rbrace, \quad x\in [0,1],$$
introduced by Feller in a related context \cite{Feller} and used by several authors, see, for example, \cite{PiermarcoGenni} or \cite{FRAGNELLI20161314} and the references therein.\\
Indeed, prior to delving into the systems discussed in this paper, a literature review on the investigation of degenerate systems would be valuable. It is well known that standard linear theory for transverse waves in a string of length $L$ under tension $\mathcal{T}$ leads to the classical wave equation
$$\rho(x)u_{tt}(x,t)=\frac{\partial \mathcal{T}}{\partial x}u_x(x,t)+\mathcal{T}(x,t)u_{xx}(x,t),
$$
where $u(x,t)$ denotes the vertical displacement of the string from the $x$ axis at position $x\in (0,L)$ and time $t> 0$, $\rho(x)$ is the mass density of the string at position $x$, while $\mathcal{T} (x,t)$ denotes the tension in the string at position $x$ and time $t$. Divide by $\rho(x)$, assume $\mathcal{T}$ is independent of $t$, and set $a(x)=\mathcal{T}(x)\rho^{-1}(x)$, $b(x)=\mathcal{T}^{\prime}(x)\rho^{-1}(x)$. In this way, we obtain
$$u_{tt}(x,t)=a(x)u_{xx}(x,t)+b(x)u_x(x,t).
$$
Let's assume that the density is remarkably high at a particular point, for example, $x = 0$. In this case, the previous equation degenerates at $x = 0$, as we can treat $a(0) = 0$, and the remaining term becomes a drift term.

Controllability problems concerning parabolic issues have become a prominent subject in contemporary research. Initially explored in the context of the heat equation, subsequent contributions have extended this investigation to encompass more generalized scenarios. A frequently adopted approach to establish controllability involves proving global Carleman estimates for the adjoint operator of the given problem. While such estimates have been extensively developed for uniformly parabolic operators without any degeneracies or singularities.\\
In recent years, researchers have expanded their investigations of these estimates to encompass operators that are not uniformly parabolic. In fact, several problems arising in Physics and Biology (see \cite{Karachalios2005}), Biology (see \cite{Boutaayamou2020,  Fragnelli+2020}), as well as Mathematical Finance (see \cite{Hagan}), are governed by degenerate parabolic equations.
The existing literature focused on controlling and stabilizing the nondegenerate wave equation using diverse damping methods is notably extensive. This fact can be observed in the substantial number of works cited, as exemplified by \cite{CHENTOUF, Conrad1993, dAndraNovel1994} and the references mentioned within.
In \cite{CHENTOUF}, the authors consider the following modelization of a flexible torque arm controlled by two feedbacks depending only on the boundary velocities:
\begin{equation*}
\begin{cases}
y_{tt}(x,t)-(ay_x)_x(x,t)+\beta y_t(x,t)+\gamma y(x,t) =0,& (x,t)\in (0,1)\times\R^+_{\ast},\\[0.1in]
(ay_x)(0,t)=\varepsilon_1U_1(t),& t>0,\\[0.1in]
(ay_x)(1,t)=\varepsilon_2U_2(t),& t>0,\\[0.1in]
\end{cases}
\end{equation*}
\text{where} $ \beta\geq 0, \, \gamma>0, \, \varepsilon_1,\varepsilon_2\geq 0, \quad \varepsilon_1+\varepsilon_2\neq 0$ and $a\in W^{1,\infty}(0,1), \, a(x)\geq a_0>0$ for all $x\in [0,1].$
They prove the exponential decay of the solutions. On the contrary, when the coefficient $a(x)$ degenerates very little is known in the literature, even though many problems that are relevant for applications are described by hyperbolic equations degenerating at the boundary of the space domain (see \cite{Gueye}).
Lately, the subject of controllability and stability in degenerate hyperbolic equations has gained significant attention, with various advancements made in recent years (see \cite{Cannarsa}, \cite{Gueye}, \cite{Zhang2015}, and the references mentioned within). In \cite{Alabau2006}, the authors prove a Carleman estimate for the one dimensional degenerate heat equation and the null controllability on $[0, 1]$ of the semilinear degenerate parabolic equation is also studied.  On the other hand,  in \cite{Fragnelli2016}, the authors establish Carleman estimates for singular/degenerate parabolic Dirichlet problems with degeneracy and singularity occurring in the interior of the spatial domain. 

More recently, Alabau et al. consider a degenerate wave equation of the form $u_{tt}-(a(x)u_x)_x=0$ in $(0,1) \times (0,+\infty)$, where $a$ is positive on $(0,1]$ and vanishes at zero. Initially presented in an arxiv preprint in May 2015 and later published in \cite{Cannarsa}, their work establishes observability inequalities for weakly and strongly degenerate equations, proving negative results when the diffusion coefficient degenerates too violently (i.e., when the constant $K$ in \eqref{Degeneracy} is greater than $2$). They also study the blow up of observability time when $K$ converges to 2 from below and  prove the exact controllability of the corresponding degenerate control problem. Using the optimal-weight convexity method of \cite{AlabauBoussouira2004} and \cite{ALABAU20101473}, together with the results for linear damping, in \cite{Cannarsa} the authors also study the boundary stabilization of the degenerate linearly damped wave equation, showing that a suitable boundary feedback stabilizes the system exponentially.  In the same work they also consider the stability analysis for the degenerate nonlinearly boundary damped wave equation for an arbitrarily growing nonlinear feedback close to the origin, showing that degeneracy does not affect optimal energy decay rates over time. Then, in \cite{Zhang2015}, Zhang and Gao focus on the case $a(x) = x^{\alpha}$  investigating null controllability for degenerate wave equations using different spectral methods. \\
In \cite{Benaissa2018}, a one-dimensional degenerate wave equation with a boundary control condition of fractional derivative type is considered.  The authors show that the problem is not uniformly stable by a spectrum method obtaining a polynomial stability.  Recently, in \cite{fragnelli2022linear},  the authors consider  a degenerate wave equation in one dimension, with drift and in presence of a leading degenerate operator which is in non divergence form.  In particular, they prove uniform exponential decay under some conditions for the  solutions of the following system
\begin{equation}\label{Genni}
\left\{\begin{array}{llll}
y_{tt}-a(x)y_{xx}-b(x)y_{x}=0,& (x,t) \in (0,1)\times(0,T),\\[0.1in]
y_t(t,1)+\eta y_x(t,1)+\beta y(t,1)=0, &t\in (0,T),\\[0.1in]
y(t,0)=0,&t\in (0,T),\\[0.1in]
y(0,x)=y_0(x), y_t(0,x)=y_1(x),&x\in (0,1),
\end{array}\right.
\end{equation}
where homogeneous Dirichlet boundary condition is taken where the degeneracy occurs and a boundary damping is considered at the other endpoint. A boundary controllability problem for a system similar to \eqref{Genni} is considered in \cite{Boutaayamou2023}. In particular, the authors study the controllability of the system  by providing some conditions for the boundary controllability of the solution of the associated Cauchy problem at a sufficiently large time.
On the other hand, in \cite{Alhabib1} the authors consider the exact boundary controllability for a degenerate and singular wave equation in a bounded interval with a moving endpoint.  Later, in \cite{Allal2022}, the boundary controllability of a one-dimensional degenerate and singular wave equation with degeneracy and singularity occurring at the boundary of the spatial domain is considered. In particular,  exact boundary controllability is proved in the range of both subcritical and critical potentials and for sufficiently large time, through a boundary controller acting away from the degenerate/singular point. \\
Some years ago, in \cite{Liu1},  the authors study the stability of an elastic string system with local Kelvin– Voigt damping given in the following system
\begin{equation}
\left\{\begin{array}{lll}
u_{tt}(x,t)-[u_{x}(x,t)+b(x)u_{xt}(x,t)]_x=0,&  t\geq 0,\, -1<x<1,\\[0.1in]
u(t,-1)=u(t,1)=0, &t\geq 0,\\[0.1in]
u(x,0)=u_0(x), u_t(x,0)=u_1(x),&x\in (-1,1).
\end{array}\right.
\end{equation}
The function $b(x)\in L^{\infty}(-1,1)$ is assumed to be
\begin{equation}
b(x)=
\left\{
\begin{array}{ll}
0,\quad &\text{for}\, x\in [-1,0),\\
a(x),\quad &\text{for}\,x\in [0,1],
\end{array}
\right.
\end{equation}
where the function $a(x)$ is nonnegative. Under the assumption that the damping coefficient has a singularity at the interface of the damped and undamped regions and behaves like $x^{\beta}$ near the interface, they prove that the semigroup corresponding to the system is polynomially (of order $\frac{1}{1-\beta}$ when $\beta\in (0,1)$) or exponentially (when $\beta\geq 1$) stable and the decay rate depends on the parameter $\beta\in (0,1]$. \\
It is known that the optimal decay rate of the solution is $t^{-2}$ in the limit case $\beta=0$ and exponential for $\beta\geq 1$. In the case when the damping coefficient $b(x)$ is continuous, but its derivative has a singularity at the interface $x = 0$, the best known decay rate is $t^{-\frac{3-\beta}{2(1-\beta)}}$(see \cite{Zhong})  , which fails to match the optimal one at $\beta = 0$.  The authors in \cite{Han2022},  obtain a sharper polynomial decay rate $t^{-\frac{2-\beta}{1-\beta}}$; more significantly, the decay rate is consistent with the optimal polynomial decay rate at $\beta = 0$  and uniform boundedness of the resolvent operator on the imaginary axis at $\beta = 1$ (consequently, the exponential decay rate at $\beta=1$ as $t\to \infty$). But, they don't reach the optimal decay rate in this case. For the case of coupling systems, we mention \cite{Wehbe2021}  where the authors prove a polynomial energy decay rate for a locally coupled wave equations with only one internal viscoelastic damping of Kelvin-Voigt type. On the other hand, in \cite{Mohammad},  the authors study the wave-Euler Bernoulli beam equations coupled through transmission with localized fractional Kelvin-Voigt damping acting on one of the two equation. In this case the authors prove polynomial energy decay rate in the different placement of the damping.\\

\noindent The main novelty of this paper lies in:\\
$\bullet$ The simultaneous consideration of degeneracy in both the wave and damping aspects within the first system. Additionally, the concept of internally localized degenerate damping remains unexplored in existing literature. Furthermore, we have achieved exponential stability regardless the of degeneracy, whether it's weak or strong,  representing an independent and significant result.  This stability is attained under assumptions that depend on the length of the damped region according to the choice of the functions $a$ and $b$. Moving on to the second system discussed in this paper, it involves a coupling of the degenerate wave equation with a non-degenerate wave equation through transmission, where only one damping effect is applied at the endpoint of the non-degenerate part.\\
$\bullet$ The degenerate wave equation has a damping term only in the first system \eqref{S1} and not in \eqref{System2}.  Actually,  in \eqref{System2} the damping occurs at the endpoint of the non-degenerate part, and the two equations are connected via transmission.  Notably, achieving exponential stability doesn't necessitate damping of the degenerate equation itself. Thus, ensuring damping at the non-degenerate equation's boundary is sufficient to establish exponential stability.\\
$\bullet$ We enhance Hypothesis 4  in \cite{fragnelli2022linear}, which leads to a more generalized condition for exponential stability compared to the one considered in that paper.\\

The paper is structured as follows. Section 1 addresses a system of a degenerate wave equation with internally local degenerate damping. We reframe the system \eqref{S1} into an evolution system and establish the well-posedness of this system using a semigroup approach. Furthermore, we demonstrate the exponential stability of the system using multiplier methods. Moving to Section 2, we investigate a degenerate wave equation coupled with a non-degenerate one via transmission. The non-degenerate wave equation is subjected to a single dissipation law, which acts only on its end.  For this particular system \eqref{System2}, the authors also prove exponential stability.

\section{Stabilization of degenerate wave equation with drift and locally internal degenerate damping}\label{sec2}
\noindent In this section, we focus on the well-posedness and the exponential stability of \eqref{S1}. 
\subsection{Preliminaries, Functional spaces and Well-Posedness}
This subsection is devoted to establish a very modest assumption and to define the functional spaces that will be used throughout the entire paper.  Moreover, the well-posedness of \eqref{S1} is studied. We begin with the following hypotheses.
\begin{Hypothesis}\label{hyp1}
Functions $a$ and $b$ are continuous in $[0,1]$ and such that $\dfrac{b}{a}\in L^1(0,1)$.
\end{Hypothesis}

\begin{Hypothesis}\label{hyp3}
Hypothesis \ref{hyp1} holds.  In addition, $a$ is such that $a(0)=0, a>0$ on $(0,1]$ and there exists $\alpha>0$ such that the function 
\begin{equation}\label{rq}
x\rightarrow\dfrac{x^{\alpha}}{a(x)}
\end{equation}
is non-decreasing in a right neighborhood of $x=0$.
\end{Hypothesis}

\begin{Remark}
\begin{enumerate}
\item If $a$ is (WD) or (SD), then \eqref{rq} holds for all $\alpha\geq K$ and for all $x\in (0,1)$.
\item We notice that, at this stage, a may not degenerate at $x = 0$.  However, if it is (WD) then $\dfrac{1}{a}\in L^1(0,1)$ and the assumption $\dfrac{b}{a}\in L^1(0,1)$ is always satisfied.  If $a$ is (SD) then $\dfrac{1}{a}\notin  L^1(0,1)$,  hence, if we want $\dfrac{b}{a}\in L^1(0, 1)$  then $b$ has to degenerate at $0$. In this case $b$ can be (WD) or (SD).
\end{enumerate}
\end{Remark}
\noindent In order to study the well-posedness of \eqref{S1}, let us recall the well-known absolutely continuous weight function
$$\eta(x):=\exp\left\lbrace\int _{\frac{1}{2}}^{x}\dfrac{b(s)}{a(s)}ds\right\rbrace, \quad x\in [0,1].$$
introduced by Feller in a related context \cite{Feller} and used later by several authors, see, for example, \cite{PiermarcoGenni}, \cite{FRAGNELLI20161314} and the references therein.

Under Hypothesis \ref{hyp1}, it is clear that the function $ \eta:[0,1]\rightarrow \R $ introduced before is well defined and we immediately find that $\eta\in C^0[0, 1] \cap C^1(0, 1]$  is a strictly positive function, which is bounded above and below by a positive constant. Notice also that $\eta$ can be extended to a function of class $C^1[0,1]$ when $b$ degenerates at $0$ not slower than $a$, for instance if $a(x) = x^K$ and $b(x)=x^m$ with $K\leq m$.\\
Now we set the function $\sigma$ as 
\begin{equation}\label{sigma}
\sigma(x):=\dfrac{a(x)}{\eta(x)},
\end{equation}
which is a continuous function in $[0, 1]$, independent of the possible degeneracy of $a$. Moreover, observe that if $u$ is a sufficiently smooth function, e.g. $u\in W^{2,1}_{loc}(0,1)$, then we can write $Bu := au_{xx} + bu_x$
as
$$Bu=\sigma(\eta u_x)_x.$$
By using the definition of $ \sigma$,  the system \eqref{S1} can be rewritten as
\begin{equation}\label{S}
\left\{\begin{array}{lll}
u_{tt}-\sigma(\eta u_{x})_{x}+h(x)u_{t}=0,& (x,t) \in (0,1)\times\R^+_{\ast},\\[0.1in]
u(t,0)=u(t,1)=,0 &t\in \R^+_{\ast},\\[0.1in]
u(0,x)=u_0(x), y(0,x)=y_0(x),&x\in (0,1).
\end{array}\right.
\end{equation}
We introduce the following Hilbert spaces
$$L_{\frac{1}{\sigma}}^2(0,1):=\left\lbrace y\in L^2(0,1); \|y\|_{\frac{1}{\sigma}}\leq\infty\right\rbrace,  \quad\langle y,z\rangle_{\frac{1}{\sigma}}:=\int _0^1 \dfrac{1}{\sigma}y\bar{z}dx,\quad \text{for every} \quad y,z\in L_{\frac{1}{\sigma}}^2(0,1),$$
$$\ma{H_{\frac{1}{\sigma}}^1(0,1)}:=L_{\frac{1}{\sigma}}^2(0,1)\cap \ma{H^1(0,1)}, \quad \langle y,z\rangle_{1}:=\langle y,z\rangle_{\frac{1}{\sigma}}+\int _0^1\eta y_x\bar{z}_xdx,\quad \text{for every}  \quad y,z\in H_{\frac{1}{\sigma}}^1(0,1), $$
and 
$$H_{\frac{1}{\sigma}}^2(0,1):=\left\lbrace y\in  H_{\frac{1}{\sigma}}^1(0,1); By\in L_{\frac{1}{\sigma}}^2(0,1) \right\rbrace\quad \langle y,z\rangle_{2}:=\langle y,z\rangle_{1}+\langle By,Bz\rangle_{\frac{1}{\sigma}}.$$
The previous inner products induce related respective norms
$$\|u\|^2_{\frac{1}{\sigma}}=\int_0^1 \frac{1}{\sigma}|u|^2dx, \quad \|u\|_{1}^2=\|u\|^2_\frac{1}{\sigma}+\int_0^1 \eta |u_x|^2dx\quad \text{and}\quad \|u\|_{2}^2=\|u\|^2_{1}+\int_0^1 \sigma |(\eta u_x)_x|^2dx.$$
Also, we consider the following spaces
$$H_{\frac{1}{\sigma},0}^1(0,1)= L_{\frac{1}{\sigma}}^2(0,1)\cap H^1_0(0,1)\quad\text{and}\quad H_{\frac{1}{\sigma},0}^2(0,1):=\left\{u\in H_{\frac{1}{\sigma},0}^1(0,1); By \in L_{\frac{1}{\sigma}}^2(0,1) \right\}$$
endowed with the previous inner products and related norms and we denote by $\|\cdot\|=\|\cdot\|_{L^2(0,1)}$.

\begin{Proposition}(Hardy-Poincar\'e Inequality)\label{prop1}
Assume Hypothesis \ref{hyp3}.  Then there exists $C_{HP} > 0$ such that
\begin{equation}\label{HP}\tag{$\rm{HP}$}
\int_0^1 u^2 \dfrac{1}{\sigma}dx\leq C_{HP} \int_0^1 u_x^2dx\quad\forall\, v\in H^1_{\frac{1}{\sigma},0}(0,1).
\end{equation}
where $\displaystyle{C_{HP}=\left(\frac{4}{a(1)}+\max_{x\in [\beta,1]}\left(\frac{1}{a}\right)C_P\right)\max_{x\in [0,1]}\eta(x)}$ ,  $C_P$ the constant of the classical Poincar\'e inequality on $(0,1)$ and $\beta\in (0,1)$.  \end{Proposition}
\begin{proof}
The one dimensional Hardy inequality with one-sided boundary condition is represented by 
$$
\abs{u(1)}^p+\frac{p-1}{p}\int_0^1\frac{\abs{u(x)}^p}{x^p}dx\leq \left(\frac{p}{p-1}\right)^{p-1}\int_0^1\abs{u_x}^2dx,
$$
for every $u\in C_{c}^{\infty}((0,1])$ and $1<p<\infty$. When $u(1)=0$, this is a well-known Hardy inequality. Now, taking $p=2$ in the above inequality, we get 
\begin{equation}\label{1Hardy}
\int_0^1\frac{\abs{u(x)}^2}{x^2}dx\leq 4\int_0^1\abs{u_x}^2dx.	
\end{equation}
Now, take $u\in H^1_{\frac{1}{\sigma}}(0,1)$ and using the definition of $\sigma$, we get 
\begin{equation}\label{2Hardy}
\int_0^1\frac{\abs{u}^2}{\sigma}dx=\int_0^1\eta \frac{\abs{u}^2}{a}dx\leq \max_{x\in [0,1]}\eta(x)\int_0^1\frac{\abs{u}^2}{a}dx.
\end{equation}
Thus it is sufficient to estimate $\displaystyle{\int_0^1\frac{\abs{u}^2}{a}dx}$. To this aim,  thanks to Hypothesis \ref{hyp3} and \eqref{1Hardy}, we get
$$
\begin{array}{l}
\displaystyle 
\int_0^1\frac{1}{a}\abs{u}^2dx=\int_0^\beta\frac{1}{a}\abs{u}^2dx+\int_\beta^1\frac{1}{a}\abs{u}^2dx\leq \frac{1}{a(1)}\int_0^1\frac{1}{x^K}\abs{u}^2dx+\int_\beta^1\frac{1}{a}\abs{u}^2dx\\
\displaystyle 
\leq \frac{1}{a(1)}\int_0^1\frac{1}{x^2}\abs{u}^2dx+\int_\beta^1\frac{1}{a}\abs{u}^2dx\leq \frac{4}{a(1)}\int_0^1\abs{u_x}^2dx+\max_{x\in [\beta,1]}\left(\frac{1}{a(x)}\right)\int_0^1\abs{u}^2dx\\
\displaystyle 
\leq \left(\frac{4}{a(1)}+\max_{x\in [\beta,1]}\left(\frac{1}{a}\right)C_P\right)\int_0^1\abs{u_x}^2dx.
\end{array}
$$
Combining the above inequality with \eqref{2Hardy}, we get the desired result \eqref{HP}.  	
\end{proof}

\noindent Using Hypothesis \ref{hyp3}, we have 
$$\int_0^1\eta |u_x|^2dx\leq \|u\|^2_{1}\leq  \left(C_{HP}\max_{x\in[0,1]}\left(\frac{1}{\eta}\right)+1\right)\int_0^1\eta |u_x|^2dx,$$
Thus,  $\ma{\|u\|^2_{\frac{1}{\sigma},0}}$  and $\displaystyle\int_0^1\eta |u_x|^2dx$ are equivalent. Moreover, the norms $\|u\|^2_{\frac{1}{\sigma},1}$ and the usual norm in $H^1_0$ i.e. $\displaystyle\int_0^1 |u_x|^2dx$ are equivalent for all $u \in H^1_0(0, 1)$. Indeed,
$$
\min_{[0,1]}\eta\int_0^1 |u_x|^2dx\leq \|u\|^2_{1}\leq \left(C_{HP}+\max_{x\in[0,1]}\eta\right)\int_0^1 |u_x|^2dx$$
where $C_{HP}$ is the Hardy-Poincar\'e constant introduced in Proposition \ref{prop1}.
Now, defining the domain of the operator $B$ as
$$D(B)=H_{\frac{1}{\sigma},0}^2(0,1)$$
and  using a semigroup approach, we will establish the well-posedness result for  \eqref{S1}. Let $(u,u_t) $ be a regular solution of the system \eqref{S}. The energy of the system is given by
\begin{equation}\label{Energy}
E(t)=\frac{1}{2}\int_0^1\left(\dfrac{1}{\sigma}|u_t|^2+\eta|u_x|^2\right)dx
\end{equation}
and we obtain that
$$
\dfrac{d}{dt} E(t)=-\int_0^1 \frac{1}{\sigma} h(x)| u_t|^2dx\leq 0.
$$
Thus, the system \eqref{S} is dissipative in the sense that its energy is a non increasing function with respect to the time variable $t$. 
We define the energy  Hilbert space $\mathcal{H}$ by 
$$
\mathcal{H}=H^1_{\frac{1}{\sigma},0}(0,1)\times L^2_{\frac{1}{\sigma}}(0,1)
$$
equipped with the following inner product
$$
\left<U,\widetilde{U}\right>_{\mathcal{H}}=\int_0^1\left(\eta u_x{\bar{\tilde{u_x}}}+\frac{1}{\sigma} v\bar{\tilde{v}}\right)dx,\quad \text{for all}\,\, U=(u,v)^{\top}\in\mathcal{H}\,\, \text{and}\,\, \tilde{U}=(\tilde{u},\tilde{v})^{\top}\in\mathcal{H}
$$
 and endowed with the the norm $\displaystyle\|U\|^2_{\mathcal{H}}=\int_0^1 \left(\eta|u_x|^2+\frac{1}{\sigma}|v|^2\right)dx$ . Finally, defining the unbounded linear operator $\mathcal{A}$ by 
$$
\mathcal{A}(u,v)^{\top}=\left(v,\,\displaystyle{\sigma(\eta  u_x)_x-h(x) v} \right)^{\top}
$$
for all $U=(u,v)\in D(\mathcal{A})$, where
$$
D(\mathcal{A}):=H^2_{\frac{1}{\sigma},0}(0,1)\times H^1_{\frac{1}{\sigma},0}(0,1),$$
we can rewrite \eqref{S} as the following evolution equation
\begin{equation}\label{Cauchy}
U_t=\mathcal{A}U,\quad U(0)=U_0
\end{equation}
where $U_0=\left(u_0,u_1\right)^{\top}$.\\
In order to estimate some terms in the following results, the below lemmas are important.

\begin{Lemma}\label{Lemma0}(See Lemma 2.2 in \cite{fragnelli2022linear})
\begin{enumerate}
\item Assume  Hypothesis \ref{hyp1}. If $u\in H^2_{\frac{1}{\sigma}}(0,1)$ and if $v\in H^1_{\frac{1}{\sigma},0}(0,1)$, then \\ $\displaystyle\lim_{x\to 0}v(x)u_x(x)=0$.
\vspace{0.2cm}
\item Assume  Hypothesis \ref{hyp3}. If $u\in D(\mathcal{A})$, then $xu_x(\eta u_x)_x\in L^1(0,1)$.
\vspace{0.2cm}
\item Assume Hypothesis \ref{hyp3}. If $u\in D(\mathcal{A})$ and $K\leq 1$, then $\displaystyle \lim_{x\to 0}x |u_x|^2=0$.

\item Assume Hypothesis \ref{hyp3}. If $u\in D(\mathcal{A})$,  $K> 1$ and $\dfrac{xb}{a}\in L^{\infty}(0,1)$, then  $\displaystyle \lim_{x\to 0}x |u_x|^2=0$.
\item Assume Hypothesis \ref{hyp3}. If $u\in H^1_{\frac{1}{\sigma}}(0,1)$, then $\displaystyle\lim_{x\to 0}\dfrac{x}{a}|u(x)|^2=0$.
\end{enumerate}
\end{Lemma}

\begin{Proposition}\label{M-Dissipative}
The unbounded linear operator $\mathcal{A}$ is m-dissipative in the energy space $\mathcal{H}$. 
\end{Proposition}
\begin{proof}
	For all $U=(u,v)^{\top}\in D(\mathcal{A})$,  we have 
	\begin{equation}\label{dissipation}
	\Re\left(\mathcal{A}U,U\right)_{\mathcal{H}}=-\int_{0}^1\dfrac{h(x)}{\sigma(x)}|\lambda u|^2dx\leq 0,
	\end{equation}
	which implies that $\mathcal{A}$ is dissipative. Now, let $F=(f_1,f_2)^{\top}\in\mathcal{H}$, we prove the existence of  $U=(u,v)^{\top}\in D(\mathcal{A})$ unique solution of the equation
	\begin{equation}\label{dissipative1}
	-\mathcal{A}U=F.
	\end{equation}
	Equivalently, we have the following system
	$$
	-v=f_1\quad \text{and}\quad -\sigma(\eta  u_x)_x+h(x) v=f_2.
	$$
	Combining the above two equations, we get 
	\begin{eqnarray}
-\sigma(\eta  u_x)_x	&=&f_2+h(x)f_1 .\label{dissipative4}
	\end{eqnarray}
Let $\varphi\in H_{\frac{1}{\sigma},0}^1(0,1)$.	Multiplying  \eqref{dissipative4} by  $\dfrac{1}{\sigma}\bar{\varphi}$  and integrate over $(0,1)$, we obtain	
\begin{equation}\label{VP}
	\Lambda(u,\varphi)=L(\varphi),\quad \forall \,\varphi\in H_{\frac{1}{\sigma},0}^1(0,1),
	\end{equation}
where
\begin{equation*}\label{coercive}
\Lambda(u,\varphi)=\int_0^1 \eta u_x\bar{\varphi}_x dx\quad \text{and}\quad L(\varphi)=\int_0^1 \frac{1}{\sigma}f_2 \bar{\varphi}dx+\int_0^1  \frac{1}{\sigma}h(x)f_1 \bar{\varphi} dx.
\end{equation*}
We have that $\Lambda$ is a sesquilinear,  continuous and coercive form on $H_{\frac{1}{\sigma},0}^1(0,1)$,  and $L$ is a continuous form on $H_{\frac{1}{\sigma},0}^1(0,1)$. Then, using Lax-Milgram Theorem, we deduce that there exists $u\in H_{\frac{1}{\sigma},0}^1(0,1)$ unique solution of the variational problem \eqref{VP}.  
Now, taking $v:=-f_1$, we have $v\in H_{\frac{1}{\sigma},0}^1(0,1)$. It remains to prove that $U \in D(\mathcal{A})$ and solves \eqref{dissipative1}. To this aim observe that equation \eqref{VP} holds for every $z\in C_c^{\infty}(0,1)$, thus we have $-(\eta  u_x)_x	=\dfrac{1}{\sigma} (f_2+h(x)f_1)$ a.e. in $(0,1)$. This implies that $-\sigma(\eta  u_x)_x	=f_2+h(x)f_1\in L^2_{\frac{1}{\sigma}}(0,1)$, i.e.  $Bu\in L^2_{\frac{1}{\sigma}}(0,1) $. Thus,  $U\in D(\mathcal{A})$. Therefore, $(u,v)\in D(\mathcal{A})$ is the unique solution of \eqref{dissipative1}. Then, $\mathcal{A}$ is an isomorphism and since $\rho(\mathcal{A})$ is open set of $\mathbb{C}$ (see Theorem 6.7 (Chapter III) in \cite{Kato01}), we easily get $R(\lambda I-\mathcal{A})=\mathcal{H}$ for a sufficiently small $\lambda>0$. This, together with the dissipativeness of $\mathcal{A}$, imply that $D(\mathcal{A})$ is dense in $\mathcal{H}$ and $\mathcal{A}$ is m-dissipative in $\mathcal{H}$ (see Theorem 4.5, 4.6 in \cite{Pazy01}). The proof is thus complete.  
\end{proof}

According to Lumer-Phillips Theorem (see \cite{Pazy01}), Proposition \ref{M-Dissipative} implies that the operator $\mathcal{A}$ generates a $C_0-$semigroup of contractions $(\mathcal T(t) )_{t \ge 0} = (e^{t\mathcal{A}})_{t \ge0}$ in $\mathcal{H}$ which gives the well-posedness of \eqref{Cauchy}. Then, we have the following result 
\begin{theoreme}
	For any $U_0\in\mathcal{H}$, problem \eqref{Cauchy} admits a unique weak solution satisfying 
	$$
	U(t)\in C^0(\R^+;\mathcal{H}). 
	$$
	Moreover, if  $U_0\in D(\mathcal{A}) $, \eqref{Cauchy} admits a unique strong solution $U$ satisfying
	$$
	U(t)\in C^1(\R^+,\mathcal{H})\cap C^0(\R^+,D(\mathcal{A})).
	$$
\end{theoreme}

\subsection{Exponential Stability}
In this subsection we prove the exponential stability of the system \eqref{S1}  when $a$ is \eqref{WD} or \eqref{SD}. Here we define the interval $I_{j\epsilon}$ such that $j\in\{0, 1,2\}$, given by $I_{j\epsilon}:=(x_1+j\epsilon,x_2-j\epsilon)$ where $\epsilon$ is such that $\epsilon< \dfrac{x_2-x_1}{4}$. In particular, we denote by $I_0=I_{0\epsilon}=(x_1,x_2)$ and $I_{\epsilon}=I_{1\epsilon}$.  Clearly, $I_{2\epsilon} \subset I_\epsilon \subset I_0$. Moreover, for convenience, we denote by
\begin{equation}\label{M}
\left\{\begin{array}{ll}
\displaystyle
M_{0,1}:=\left\|x\frac{b}{a}\right\|_{L^{\infty}(0,x_1+2\epsilon)}\\
\displaystyle
M_{0,2}:=\left\|x\frac{a^{\prime}-b}{a}\right\|_{L^{\infty}(0,x_1+2\epsilon)}
\end{array}\right.
\text{and}\quad
\left\{\begin{array}{ll}
\displaystyle
M_{1,1}:=\left\|(x-1)\frac{b}{a}\right\|_{L^{\infty}(x_2-2\epsilon,1)}\\
\displaystyle
M_{1,2}:=\left\|(x-1)\frac{a^{\prime}-b}{a}\right\|_{L^{\infty}(x_2-2\epsilon,1)}.
\end{array}\right.
\end{equation}
\begin{Hypothesis}\label{hyp2}
Assume Hypothesis \ref{hyp1}, $a$  \eqref{WD} or \eqref{SD} and the functions $a$ and $b$ such that 
\center{$M_{0,1}+M_{1,1}<1-\frac{K}{2}$ and $M_{0,2}+M_{1,2} <1+\frac{K}{2} $.}
 
\end{Hypothesis}
\begin{Remark}
We remark that the choice of the functions $a$ and $b$ gives reliance on the length of the damped interval; i.e. the choice of $x_1$ and $x_2$ depends on the choice of the functions $a$ and $b$.
\end{Remark}
\begin{ex}\normalfont
$(1)$ Example for the (WD) case: $a(x)=\sqrt{x}$ and $b(x)=1$.  In order for  the functions $a$ and $b$ to satisfy Hypothesis \ref{hyp2}, we need
$$\sqrt{x_1+2\epsilon}+\frac{1-(x_2-2\epsilon)}{\sqrt{x_2-2\epsilon}}<\frac{3}{4}\quad \text{and}\quad \left\|(x-1)\left(\frac{1}{2x}-\frac{1}{\sqrt{x}}\right)\right\|_{L^{\infty}(x_2-2\epsilon,1)}<\frac{3}{4}.$$ 
So, it is sufficient to take $x_2-2\epsilon>\frac{41-3\sqrt{73}}{32}\approx 0.4802$ and  $0<x_1<x_1+2\epsilon<x_2-2\epsilon<x_2<1$ in order to satisfy the above inequalities.
(As a specific choice for $x_1$ and $x_2$, we take $x_1+2\epsilon=0.15$ and $x_2-2\epsilon=0.8$, see the below Figure \ref{f1} that corresponds to this specific choice).
\\
$(2)$ Example for the \eqref{SD} case: $a(x)=x\sqrt{x}$ and $b(x)=\frac{1}{8}x$. In order for  the functions $a$ and $b$ to satisfy Hypothesis \ref{hyp2}, we need
$$\sqrt{x_1+2\epsilon}+\frac{1-(x_2-2\epsilon)}{\sqrt{x_2-2\epsilon}}<2\quad\text{and}\quad \left\|(x-1)\left(\frac{3}{2x}-\frac{1}{8\sqrt{x}}\right)\right\|_{L^{\infty}(x_2-2\epsilon,1)}<\frac{1}{4}.
$$
So, it is sufficient to take $x_2-2\epsilon>0.8471$ and  $0<x_1<x_1+2\epsilon<x_2-2\epsilon<x_2<1$ in order to satisfy the above inequalities.
(As a specific choice here , we take $x_1+2\epsilon=0.3$ and $x_2-2\epsilon=0.9$, see the below Figure \ref{f2} that corresponds to this specific choice).
\end{ex}
\begin{figure}[h!]
  \centering
  \begin{subfigure}[b]{0.45\linewidth}
    \includegraphics[width=\linewidth]{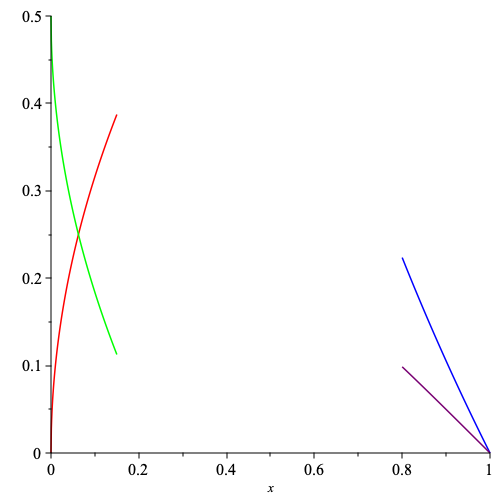}
    \caption{$a(x)=\sqrt{x},b(x)=1$}\label{f1}
  \end{subfigure}
  \begin{subfigure}[b]{0.45\linewidth}
    \includegraphics[width=\linewidth]{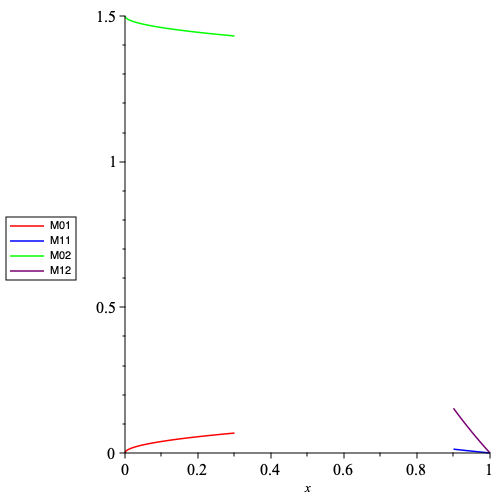}
    \caption{$a(x)=x\sqrt{x}, b(x)=\frac{x}{8}$}\label{f2}
  \end{subfigure}
  \caption{Examples for (WD) and (SD) cases}
  \label{fig:coffee}
\end{figure}
where M01, M11, M02,  and M12 represent the functions inside the norms given in \eqref{M}.

\begin{Theorem}\label{Exponential Stab}
Assume Hypothesis \ref{hyp2}. Then,  the $C_0-$semigroup of contractions $\left(\mathcal T(t)\right)_{t\geq 0}$ is exponentially  stable, i.e. there exist constants $M\geq 1$ and $\tau>0$ independent of $U_0$ such that 
$$
\left\|\mathcal T(t)_0\right\|_{\mathcal{H}}\leq Me^{-\tau t}\|U_0\|_{\mathcal{H}},\qquad t\geq 0.
$$
\end{Theorem}
\noindent According to Huang \cite{Huang01} and Pruss \cite{pruss01}, we have to check if the following conditions hold:
\begin{equation}\label{H1}\tag{${\rm H1}$}
 i\mathbb{R}\subseteq \rho\left(\mathcal{A}\right)
\end{equation}
and 
\begin{equation}\label{H2}\tag{${\rm H2}$}
\displaystyle{\sup_{\la\in \R}\|\left(i\la I-\mathcal{A}\right)^{-1}\|_{\mathcal{L}\left(\mathcal{H}\right)}=O(1).}
\end{equation}
The following proposition is a technical finding that will be used to prove Theorem \ref{Exponential Stab}.
\begin{Proposition}\label{prop2}
Assume Hypothesis \ref{hyp2} and let $(\lambda,U:=(u,v))\in \R^{\ast}\times D(\mathcal{A})$, with $\lambda\neq 0$, such that
\begin{equation}\label{eq0}
\left(i\la I-\mathcal{A}\right)U=F:=(f^1,f^2)\in \mathcal{H} ,
\end{equation}
i.e.
\begin{equation}\label{eq1}
i\la u-v=f^1\quad \text{and}\quad i\la v-\sigma(\eta u_x)_x+h(x)v=f^2.	
\end{equation}
Then,  we have the following inequality
\begin{equation}\label{prop2.9}
\|U\|_{\mathcal{H}}\leq \kappa \left(1+\frac{1}{\la^2}\right)  \|F\|_{\mathcal{H}},
\end{equation}
where $\kappa$ is a suitable positive constant independent of $\lambda$.
\end{Proposition}

\noindent  Observe that,  by substituting $v=i\la u-f^1$  into the second equation in \eqref{eq1}, we get
\begin{equation}\label{Eq1}
\la ^2 u +\sigma(\eta u_x)_x-h(x)i\la u=-\left(f^2+i\la f^1+h(x)f^1\right).
\end{equation}
Using equation \eqref{eq1} and the Hardy-Ponicar\'e inequality given in Proposition \ref{prop1},  we get
\begin{equation}\label{lambdau}
\|\la u\|_{\frac{1}{\sigma}}\leq \|v\|_{\frac{1}{\sigma}}+\sqrt{C_{HP}}\|f^1_x\|
\leq \max\{1, c_0\}\left(\|v\|_{\frac{1}{\sigma}}+\|\sqrt{\eta}f_x^1\|\right)
\leq c_1 \left[\|U\|_{\mathcal{H}}+\|F\|_{\mathcal{H}}\right]
\end{equation}
where $c_1=\max\left(1,c_0\right)$ with $c_0=\sqrt{C_{HP}\max\limits_{x\in[0,1]}\eta^{-1}(x)}$.\\
Moreover, consider the following cut-off functions:
given a function $\theta$ in $C^\infty\left([0,1]\right)$ such that
\begin{equation}\label{theta}
0\leq\theta\leq 1,\quad \theta=1\ \ \text{on}\ \ I_{2\epsilon},\quad \text{and}\quad \theta=0\ \ \text{on}\ \ (0,1)\backslash I_{\epsilon},
\end{equation}
define
the functions $\varphi_1$ and $\varphi_2$ in $C^{\infty}([0,1])$ so that $0\leq\varphi_1, \varphi_2\leq 1$,

\begin{equation}\label{varphi12}
\varphi_1(x):=
\left\lbrace
\begin{array}{ccc}
1&\text{on}&[0, x_1+2\epsilon]\\
0&\text{on}&[ x_2-2\epsilon,1]
\end{array}
\right.
\quad
\text{and}
\quad
\varphi_2(x)=
\left\lbrace
\begin{array}{ccc}
0&\text{on}&[0, x_1+2\epsilon]\\
1&\text{on}&[x_2-2\epsilon,1].
\end{array}
\right.
\end{equation}
Finally, define $\varphi(x)=x\varphi_1(x)+(x-1)\varphi_2(x)$, $x\in [0,1]$.
Now, we introduce the following lemma that will be used in the proof of Proposition \ref{prop2}.
\begin{Lemma}\label{Lemma01}
\begin{enumerate}
\item Assume Hypothesis \ref{hyp3}. If $u\in D(\mathcal{A})$ and $K\leq 1$, then $\displaystyle \lim_{x\to 0}\eta\varphi |u_x|^2=0$.

\item Assume Hypothesis \ref{hyp3}. If $u\in D(\mathcal{A})$,  $K> 1$ and $\dfrac{xb}{a}\in L^{\infty}(0,1)$, then  $\displaystyle \lim_{x\to 0}\eta\varphi |u_x|^2=0$.
\item Assume Hypothesis \ref{hyp3}. If $u\in H^1_{\frac{1}{\sigma}}(0,1)$, then $\displaystyle\lim_{x\to 0}\dfrac{\varphi(x)}{\sigma(x)}|u(x)|^2=0$.

\item Assume Hypothesis \ref{hyp3}. If $u, f^1\in H^1_{\frac{1}{\sigma}}(0,1)$, then $\displaystyle\lim_{x\to 0}\Re\left[2i\la f^1\frac{\varphi}{\sigma}\bar{u}\right]=0.$
\end{enumerate}
\end{Lemma}
\begin{proof}
For the proof of the first three items in the Lemma we rely on the definition of $\varphi$ defined below and the proof of Lemma 2.2 in \cite{fragnelli2022linear}.\\
For the last item (4), using the definition of $\varphi$, we have that  $\displaystyle\lim_{x\to 0}\Re\left[2i\la f^1\frac{\varphi}{\sigma}\bar{u}\right]=\lim_{x\to 0}\Re\left[2i\la f^1\frac{x\varphi_1}{\sigma}\bar{u}\right]$. So, by using Young's inequality we get
\begin{equation}
\left|\Re\left[2i\la f^1\frac{\varphi}{\sigma}\bar{u}\right]\right|\leq \frac{x|\la u|^2}{\sigma}|\varphi_1|+\frac{x|f^1|^2}{\sigma}|\varphi_1|.
\end{equation}
Then,  by using items (5) of Lemma \ref{Lemma0}, we obtain  $\displaystyle\lim_{x\to 0} \left|\Re\left[2i\la f^1\frac{\varphi}{\sigma}\bar{u}\right]\right|=0$ and thus the proof is complete.
\end{proof}
\textbf{Proof of Proposition \ref{prop2}}. We divide the proof of Proposition  \ref{prop2}  into several steps.\\
\underline{\textbf{Step 1.}} The aim of this step is to show that the solution $U = (u, v)\in D(\mathcal{A})$ of equation \eqref{eq0} satisfies the following two estimates
\begin{equation}\label{stab1}
\int _{I_{\epsilon}}\dfrac{1}{\sigma}|\la u|^2dx \leq \kappa_1\left[\|U\|_{\mathcal{H}}\|F\|_{\mathcal{H}}+\|F\|_{\mathcal{H}}^2\right]\quad \text{and}\quad \int_{I_{2\epsilon}}\eta | u_x|^2dx \leq \kappa_2 \left[\|U\|_{\mathcal{H}}\|F\|_{\mathcal{H}}+\|F\|_{\mathcal{H}}^2\right],
\end{equation}
where $\kappa_1$ and $\kappa_2$ are constants to be determined.\\
Firstly, taking the inner product of $F$ with $U$ in $\mathcal{H}$, using \eqref{dissipation} and Cauchy-Schwarz inequality,  we get
\ibt{\begin{equation}\label{eq4}
	\int_{0}^1\dfrac{h(x)}{\sigma(x)}|v|^2dx=\Re\left(\left\langle-\mathcal{A}U,U\right\rangle_{\mathcal{H}}\right)=\Re\left(\left\langle F,U\right\rangle_{\mathcal{H}}\right)\leq \|U\|_{\mathcal{H}}\|F\|_{\mathcal{H}}.
	\end{equation}
	Using the first equation in \eqref{eq1}, and using the above equation, Young and \eqref{HP} inequalities, we have
	
		\begin{align}
	&\int_{0}^1\dfrac{h(x)}{\sigma(x)}|\lambda u|^2dx \leq 2\int_{0}^1\dfrac{h(x)}{\sigma(x)}|v|^2dx+2\int_{0}^1\dfrac{h(x)}{\sigma(x)}|f^1|^2dx\leq 2\|U\|_{\mathcal{H}}\|F\|_{\mathcal{H}}+2\max _{[x_1,x_2]}h_1(x)\int_0^1 \frac{1}{\sigma}|f^1|^2dx\\
&	\leq 2\|U\|_{\mathcal{H}}\|F\|_{\mathcal{H}}+2\max _{[x_1,x_2]}h_1(x)c_0^2\int_0^1 \eta |f^1|^2dx\leq \widetilde{c_0} \left[\|U\|_{\mathcal{H}}\|F\|_{\mathcal{H}}+\|F\|_{\mathcal{H}}^2\right],\nonumber
	 		\end{align}
where $\widetilde{c_0}=2\max\{1, \max\limits _{[x_1,x_2]}h_1(x)c_0^2\}$.
Using the fact that $h_1(x)\neq 0$ in $I_{\epsilon}$ and $h_1\in L^{\infty}(I_{\epsilon})$, we have $\min\limits_{x\in\overline{ I_{\epsilon}}}h_1(x)\leq h_1(x) \leq \max\limits_{x\in \overline{I_{\epsilon}}}h_1(x)$.
	Thus, 
	$$\int_{I_{\epsilon}}\dfrac{h_1(x)}{\sigma}|\la u|^2dx\leq \widetilde{c_0}[ \|U\|_{\mathcal{H}}\|F\|_{\mathcal{H}}+\|F\|^2_{\mathcal{H}}]\quad\Longrightarrow \quad \int _{I_{\epsilon}}\dfrac{1}{\sigma}|\la u|^2dx     \leq\kappa_1\left[\|U\|_{\mathcal{H}}\|F\|_{\mathcal{H}}+\|F\|_{\mathcal{H}}^2\right]$$
	with $\kappa_1=\dfrac{\widetilde{c_0}}{\min\limits_{x\in \overline{I_{\epsilon}}}h_1(x)}$ and thus the first estimation in \eqref{stab1} is proved.\\}
Secondly, multiplying \eqref{Eq1} by $\dfrac{1}{\sigma}\theta \bar{u}$,  integrating over $(0,1)$,  and taking the real part we get,
\begin{equation}\label{E0}
\int_0^1 \dfrac{1}{\sigma}\theta |\la u|^2dx+\dfrac{1}{2}\int_0^1(\eta \theta^\prime)^\prime|u|^2dx-\int_0^1\eta\theta|u_x|^2dx=-\Re\left(\int_0^1\dfrac{\theta}{\sigma}\left(f^2+i\la f^1+h(x)f^1\right)\bar{u}dx\right).
\end{equation}
In order to estimate the terms in the previous equality we use the first inequality in \eqref{stab1},  the definition of the function $\theta$ and the fact that $\supp \theta= \overline{I_{\epsilon}}$,  the Cauchy-Schwarz inequality and \eqref{lambdau}  obtaining
\begin{equation}\label{E1}
\int_0^1\dfrac{1}{\sigma}\theta|\la u|^2dx\leq \int_{I_{\epsilon}}\dfrac{1}{\sigma}|\la u|^2dx\leq \kappa_1\left[\|U\|_{\mathcal{H}}\|F\|_{\mathcal{H}}+\|F\|_{\mathcal{H}}^2\right],
\end{equation}

\begin{equation}\label{E2}
\left|\dfrac{1}{2}\int_0^1(\eta \theta^\prime)^\prime|u|^2dx\right|\leq \dfrac{1}{2}\left\|(\eta \theta^\prime)^\prime\sigma\right\|_{L^{\infty}(I_{\epsilon})}\int _{I_{\epsilon}}\dfrac{1}{\sigma}|u|^2dx\leq \dfrac{c_2}{\la^2}\left[\|U\|_{\mathcal{H}}\|F\|_{\mathcal{H}}+\|F\|_{\mathcal{H}}^2\right],
\end{equation}

\begin{equation}\label{E4}
\left|\int_0^1 \dfrac{i\la\theta}{\sigma}f^1\bar{u}dx\right|\leq \int_0^1\frac{1}{\sqrt{\sigma}}|f^1|\frac{1}{\sqrt{\sigma}}|\la u|dx
\leq c_3  \left[\|U\|_{\mathcal{H}}\|F\|_{\mathcal{H}}+\|F\|_{\mathcal{H}}^2\right],
\end{equation}
where $\displaystyle c_2=\dfrac{1}{2}\|(\eta \theta^\prime)^\prime\sigma\|_{L^{\infty}(I_{\epsilon})}\kappa_1$, and $c_3=c_0c_1$.
Now, by \eqref{HP}, we get

\begin{equation}\label{E3}
\left|\int_0^1\dfrac{\theta}{\sigma}f^2\bar{u}dx\right|\leq \left(\int_0^1 \dfrac{1}{\sigma}|f^2|^2dx\right)^{1/2}\left(\int_0^1 \dfrac{1}{\sigma}|u|^2dx\right)^{1/2}\leq c_0 \left[\|U\|_{\mathcal{H}}\|F\|_{\mathcal{H}}+\|F\|_{\mathcal{H}}^2\right],
\end{equation}
 and
\begin{equation}\label{E5}
\left| \int_0^1\dfrac{h(x)}{\sigma}\theta f^1\bar{u}\right| \leq \max\limits_{x\in [x_1,x_2]} h(x)\int_0^1 \frac{1}{\sqrt{\sigma}}|f^1| \frac{1}{\sqrt{\sigma}}|u|dx\leq  c_4\left[\|U\|_{\mathcal{H}}\|F\|_{\mathcal{H}}+\|F\|_{\mathcal{H}}^2\right],
\end{equation}
where $c_4=c_0^2\,\max\limits_{x\in [x_1,x_2]}h(x)$.
Thus, using equations  \eqref{E1}-\eqref{E5} in \eqref{E0}, we get
\begin{equation}
\int_{I_{2\epsilon}}\eta\theta | u_x|^2dx \leq \kappa_2 \left[\|U\|_{\mathcal{H}}\|F\|_{\mathcal{H}}+\|F\|_{\mathcal{H}}^2\right]
\end{equation}
where $\kappa_2=\kappa_2^{\prime}+\dfrac{c_2}{\la ^2}$ with $\kappa_2^{\prime}=\kappa_1+c_0+c_3+c_4$. Hence, thanks to the definition of $\theta$ we get the second estimate in  \eqref{stab1}.

\underline{\textbf{Step 2.}} The aim of this step is to show that the solution $U = (u, v)\in D(\mathcal{A})$ of \eqref{eq0} satisfies the following equation
\begin{equation}\label{e2}
\begin{array}{l}
\displaystyle\int_0^1 \frac{\varphi^{\prime}}{\sigma}|\la u|^2dx+\int_0^1 \eta \varphi^{\prime}|u_x|^2dx
=\int_0^1\frac{\varphi}{\sigma}\left(\frac{a^\prime-b}{a}\right)|\la u|^2dx+\int_0^1\varphi\frac{b}{a}\eta|u_x|^2dx-\lim_{x\to 0}\eta\varphi |u_x|^2
\\ 
\displaystyle
-\la^2\lim_{x\to 0}\dfrac{\varphi(x)}{\sigma(x)}|u(x)|^2
-2\Re\left(\int_0^1 \dfrac{h(x)}{\sigma}i\la u\varphi\bar{u}_xdx\right)+2\Re\left(\int_0^1\left(f^2+h(x)f^1\right)\dfrac{\varphi}{\sigma}\overline{u_x}dx\right)
\\ 
\displaystyle-2\Re\left(i\int_0^1 \left(\frac{f^1\varphi}{\sigma}\right)_x\la \bar{u}dx\right)-2\lim_{x\to 0}\Re\left[i\la f^1(x)\frac{\varphi(x)}{\sigma(x)}\bar{u}(x)\right].
\end{array}
\end{equation}
First, multiplying  \eqref{Eq1} by $-2\dfrac{\varphi}{\sigma}\bar{u}_x$, integrating over $(0,1)$ and taking the real part,   we get
\begin{equation}\label{e1}
\begin{array}{l}
\displaystyle\int_0^1 \left(\frac{\varphi}{\sigma}\right)^{\prime}|\la u|^2dx+\lim_{x\to 0}\dfrac{\varphi(x)}{\sigma(x)}|\la u(x)|^2-2\Re\left(\int_0^1 (\eta u_x)_x\varphi \bar{u}_xdx\right)+2\Re\left(\int_0^1 \dfrac{h(x)}{\sigma}i\la u\varphi\bar{u}_xdx\right)\\ \displaystyle =2\Re\left(\int_0^1\left(f^2+h(x)f^1\right)\dfrac{\varphi}{\sigma}\overline{u_x}dx\right)-2\Re\left(i\int_0^1 \left(\frac{f^1\varphi}{\sigma}\right)_x\la \bar{u}dx\right)-2\lim_{x\to 0}\Re\left[i\la f^1(x)\frac{\varphi(x)}{\sigma(x)}\bar{u}(x)\right].

\end{array}
\end{equation}
For the first term in the above equation, we have that $\displaystyle\left(\frac{\varphi}{\sigma}\right)^{\prime}=\frac{\varphi^{\prime}}{\sigma}-\frac{\varphi}{\sigma}\left(\frac{a^\prime-b}{a}\right)$ and $\eta'=\dfrac{b}{a}\eta$,  
then 
\begin{equation*}
 \int_0^1 \left(\frac{\varphi}{\sigma}\right)^{\prime}|\la u|^2dx=\int_0^1 \frac{\varphi^{\prime}}{\sigma}|\la u|^2dx-\int_0^1\frac{\varphi}{\sigma}\left(\frac{a^\prime-b}{a}\right)|\la u|^2dx
\end{equation*}
and 
\begin{equation*}
\begin{array}{l}
\displaystyle
-2\Re\left(\int_0^1 (\eta u_x)_x\varphi \bar{u}_xdx\right)=2\Re\left(\int_0^1\eta u_x(\varphi \bar{u}_x)_xdx\right)+2\lim_{x\to 0}\eta(x)\varphi(x)|u_x(x)|^2=2\int_0^1\eta\varphi^{\prime}|u_x|^2dx 
\\ \displaystyle
-\int_0^1 (\eta\varphi)^{\prime}|u_x|^2dx
+\lim_{x\to 0}\eta(x)\varphi(x) |u_x|^2=\int_0^1 \eta\varphi^{\prime}|u_x|^2dx-\int_0^1 \varphi \frac{b}{a}\eta|u_x|^2dx+\lim_{x\to 0}\eta(x)\varphi(x) |u_x(x)|^2.
\end{array}
\end{equation*}
Substituting the above two equations into \eqref{e1}, we get \eqref{e2}.

\underline{\textbf{Step 3.}} The aim of this step is to estimate the terms on the right hand side of \eqref{e2}.\\
$\bullet$ We start by the term $\displaystyle 2\Re\left(\int_0^1 \dfrac{h(x)}{\sigma}i\la u\varphi\bar{u}_xdx\right)$. By the Cauchy-Schwarz inequality, \eqref{eq4} and the inequality: $\sqrt{p+q}\leq \sqrt{p}+\sqrt{q}$ for all $p,q\geq 0$, we get
\begin{equation}\label{E11}
\begin{array}{ll}
\displaystyle
2\left|\Re \left(\int_0^1 \frac{h(x)}{\sigma}i\la u\varphi\bar{u}_xdx\right) \right| \leq 2\max_{\overline{I}_0}\sqrt{h(x)}\|\varphi\|_{\infty} \int_{I_0}\frac{|\la|\sqrt{h(x)}}{\sqrt{\sigma}}|u|\frac{\sqrt{\eta}}{\sqrt{a}}|u_x|dx
\\
\displaystyle
\ibt{\leq  2\max_{x\in\overline{I_0}}\sqrt{h(x)}\left\|\varphi\right\|_{\infty}\max_{x\in \overline{I_0}}\left(\frac{1}{\sqrt{a}}\right)\sqrt{\widetilde{c_0}}\left(\|U\|_{\mathcal{H}}\|F\|_{\mathcal{H}}+\|F\|_{\mathcal{H}}^2\right)^{1/2}\|U\|_{\mathcal{H}}}
\\
\displaystyle
\ibt{\leq 2\max_{x\in\overline{I_0}}\sqrt{h(x)} \left\|\varphi\right\|_{\infty}\max_{x\in \overline{I_0}}\left(\frac{1}{\sqrt{a}}\right)\sqrt{\widetilde{c_0}}\left(\|U\|_{\mathcal{H}}^{3/2}\|F\|_{\mathcal{H}}^{1/2}+\|F\|_{\mathcal{H}}\|U\|_{\mathcal{H}}\right)}
\\
\displaystyle
\ibt{\leq c_5 \|U\|_{\mathcal{H}}^{3/2}\|F\|_{\mathcal{H}}^{1/2}+c_5\left[\|U\|_{\mathcal{H}}\|F\|_{\mathcal{H}}+\|F\|_{\mathcal{H}}^2\right]},
\end{array}
\end{equation}
where \ibt{$\displaystyle c_5=2\max_{x\in\overline{I_0}}\sqrt{h(x)} \left\|\varphi\right\|_{\infty}\max_{x\in \overline{I_0}}\left(\frac{1}{\sqrt{a}}\right)\sqrt{\widetilde{c_0}}$}.\\
$\bullet$ 
For the term $\displaystyle 2\Re\left(\int_0^1h(x)f^1\dfrac{\varphi}{\sigma}\overline{u_x}dx\right)$, observe that by the Cauchy-Schwarz and the Hardy-Poincar\'e inequalities we get
\begin{equation}\label{E21}
\begin{array}{ll}
\displaystyle \left|2\Re\left(\int_0^1h(x)f^1\dfrac{\varphi}{\sigma}\overline{u_x}dx\right)\right|\leq 2\max\limits_{x\in \overline{I_0}}h(x)\|\varphi\|_{\infty}\int_{I_0}\dfrac{|f^1|}{\sqrt{\sigma}}\dfrac{\sqrt{\eta}}{\sqrt{a}}|u_x|dx\\
\displaystyle\leq 2\max\limits_{x\in\overline{I_0}}h(x)\|\varphi\|_{\infty} \max\limits_{x\in\overline{I_0}}\left(\dfrac{1}{\sqrt{a}}\right)\left(\ma{\int_0^1}\dfrac{|f^1|^2}{\sigma}dx\right)^{1/2}\left(\int_{I_0}\eta |u_x|^2dx\right)^{1/2}
\leq c_6 \left[\|U\|_{\mathcal{H}}\|F\|_{\mathcal{H}}+\|F\|_{\mathcal{H}}^2\right]
\end{array}
\end{equation}
where $\displaystyle c_6=2c_0\max\limits_{x\in\overline{I_0}}h(x)\|\varphi\|_{\infty} \max\limits_{x\in\overline{I_0}}\left(\dfrac{1}{\sqrt{a}}\right)$.\\
$\bullet$ Now, consider the term $\displaystyle  2\Re\left(\int_0^1f^2\dfrac{\varphi}{\sigma}\overline{u_x}dx\right).$ By definition of $\varphi$, we have
\begin{equation}\label{ee1}
 2\Re\left(\int_0^1f^2\dfrac{\varphi}{\sigma}\overline{u_x}dx\right)= 2\Re\left(\int_0^1\dfrac{x\varphi_1}{\sigma}f^2\overline{u_x}dx\right)+2\Re\left(\int_0^1\dfrac{(x-1)\varphi_2}{\sigma}f^2\overline{u_x}dx\right).
\end{equation}
First of all, consider the first term in \eqref{ee1}. Obviously, we can rewrite it as
\begin{equation}\label{ee2}
2\Re\left(\int_0^1\dfrac{x\varphi_1}{\sigma}f^2\bar{u}_xdx\right)=2\Re\left(\int_0^{x_1+2\epsilon}\dfrac{x}{\sigma}f^2\overline{u_x}dx\right)+2\Re\left(\int_{I_{2\epsilon}}\dfrac{x\varphi_1}{\sigma}f^2\bar{u}_xdx\right).
\end{equation}
Using the Cauchy-Schwarz inequality, the fact that $a$ doesn't vanish on the interval $I_{2\epsilon}$, the monotonicity of the function $\dfrac{x}{\sqrt{a}}$ in all the interval $(0,1]$ if $a$ is \eqref{WD} or \eqref{SD}  and the fact that $K<2$,  we obtain
\begin{equation}\label{ee2,1}
\left|2\Re\left(\int_0^{x_1+2\epsilon}\dfrac{x}{\sigma}f^2\overline{u_x}dx\right)\right|\leq \frac{2(x_1+2\epsilon)}{\sqrt{a(x_1+2\epsilon})}\int_{0}^{x_1+2\epsilon}\frac{1}{\sqrt{\sigma}}|f^2| \sqrt{\eta}|{u}_x|dx\leq c_7 \left[\|U\|_{\mathcal{H}}\|F\|_{\mathcal{H}}+\|F\|_{\mathcal{H}}^2\right]
\end{equation}
and 
\begin{equation}\label{ee2,2}
\begin{array}{ll}
\displaystyle
\left|2\Re\left(\int_{I_{2\epsilon}}\dfrac{x\varphi_1}{\sigma}f^2\overline{u_x}dx\right)\right|\leq 2(x_2-2\epsilon)\max_{x\in \overline{I_{2\epsilon}}}\left(\frac{1}{\sqrt{a}}\right)\int_{I_{2\epsilon}}\frac{1}{\sqrt{\sigma}}|f^2| \sqrt{\eta}|{u}_x|dx\leq c_8  \left[\|U\|_{\mathcal{H}}\|F\|_{\mathcal{H}}+\|F\|_{\mathcal{H}}^2\right],
\end{array}
\end{equation}
where  $\displaystyle c_7=\frac{2(x_1+2\epsilon)}{\sqrt{a(x_{1+2\epsilon})}}$ and $\displaystyle c_8=2(x_2-2\epsilon)\max_{x\in \overline{I_{2\epsilon}}}\left(\frac{1}{\sqrt{a}}\right)$.
Thus, substituting equations \eqref{ee2,1} and \eqref{ee2,2} into \eqref{ee2}, we get
\begin{equation}\label{ee3}
\left|2\Re\left(\int_0^1\dfrac{x\varphi_1}{\sigma}f^2\bar{u}_xdx\right)\right|\leq c_9 \left[\|U\|_{\mathcal{H}}\|F\|_{\mathcal{H}}+\|F\|_{\mathcal{H}}^2\right]
\end{equation}
where $c_9=c_7+c_8$.\\
Now, for the second term in \eqref{ee1}, we have
\begin{equation}\label{ee4}
\left|2\Re\left(\int_0^1\dfrac{(x-1)\varphi_2}{\sigma}f^2\overline{u_x}dx\right)\right|\leq 2\int_{x_1+2\epsilon}^1\dfrac{|x-1|}{\sqrt{\sigma}}|f^2|\frac{\sqrt{\eta}}{\sqrt{a}}|\bar{u}_x|dx\leq c_{10} \left[\|U\|_{\mathcal{H}}\|F\|_{\mathcal{H}}+\|F\|_{\mathcal{H}}^2\right]
\end{equation}
where $\displaystyle c_{10}=2(1-(x_1+2\epsilon))\max\limits_{x\in [x_1+2\epsilon,1] }\left(\frac{1}{\sqrt{a}}\right)$.
Using \eqref{ee3} and \eqref{ee4} in \eqref{ee1}, we get
\begin{equation}\label{E31}
\left|2\Re\left(\int_0^1f^2\dfrac{\varphi}{\sigma}\overline{u_x}dx\right)\right|\leq c_{11}\left[\|U\|_{\mathcal{H}}\|F\|_{\mathcal{H}}+\|F\|_{\mathcal{H}}^2\right]
\end{equation}
where $\displaystyle c_{11}=c_9+c_{10}$.\\
$\bullet$ Finally, consider the last term $\displaystyle  2\Re\left(i\int_0^1 \left(\frac{f^1\varphi}{\sigma}\right)_x\la \bar{u}dx\right)$. Hence, we have
\begin{equation}\label{ee5}
2\Re\left(i\int_0^1 \left(\frac{f^1\varphi}{\sigma}\right)_x\la \bar{u}dx\right)=2\Re\left(i\int_0^1 \frac{\varphi f^1_x}{\sigma}\la \bar{u}dx\right)+2\Re\left(i\int_0^1 \left(\frac{\varphi}{\sigma}\right)^{\prime}f^1\la \bar{u}dx\right).
\end{equation}
Using the definition of $\varphi$, the first term on the right hand side of \eqref{ee5}, becomes
\begin{equation}\label{ee6}
2\Re\left(i\int_0^1 \frac{\varphi f^1_x}{\sigma}\la \bar{u}dx\right)=2\Re\left(i\int_0^{x_2-2\epsilon} \frac{x}{\sigma}\varphi_1 f^1_x\la \bar{u}dx\right)+2\Re\left(i\int_{x_1+2\epsilon}^1 \frac{(x-1)}{\sigma}\varphi_2 f^1_x\la \bar{u}dx\right).
\end{equation}
Now, we need to estimate the terms in \eqref{ee6}. Thus, using again the monotonicity of $\dfrac{x}{\sqrt{a}}$ and \eqref{lambdau}, we obtain
\begin{equation}\label{ee6,1}
\left|2\Re\left(i\int_0^{x_2-2\epsilon} \frac{x}{\sigma}\varphi_1 f^1_x\la \bar{u}dx\right)\right| 
\leq 2\int_0^{x_2-2\epsilon} \frac{x}{\sqrt{a}}|\varphi_1| \sqrt{\eta}|f^1_x|\frac{1}{\sqrt{\sigma}}|\la u|dx\leq c_{12} \left[\|U\|_{\mathcal{H}}\|F\|_{\mathcal{H}}+\|F\|_{\mathcal{H}}^2\right]
\end{equation}
and 
\begin{equation}\label{ee6,2}
\left|2\Re\left(i\int_{x_1+2\epsilon}^1 \frac{(x-1)}{\sigma}\varphi_2 f^1_x\la \bar{u}dx\right)\right|\leq 2\int_{x_1+2\epsilon}^1 \frac{|x-1|}{\sqrt{a}}\sqrt{\eta} |f^1_x|\frac{1}{\sqrt{\sigma}}|\la \bar{u}|dx\leq c_{13}  \left[\|U\|_{\mathcal{H}}\|F\|_{\mathcal{H}}+\|F\|_{\mathcal{H}}^2\right],
\end{equation}
where $\displaystyle c_{12}=2\frac{(x_2-2\epsilon)}{\sqrt{a(x_2-2\epsilon)}}c_1$ and $c_{13}=c_{10}c_1$.
Using \eqref{ee6,1} and \eqref{ee6,2} in \eqref{ee6}, we get
\begin{equation}\label{ee7}
\left|2\Re\left(i\int_0^1 \frac{\varphi f^1_x}{\sigma}\la \bar{u}dx\right)\right|\leq c_{14}  \left[\|U\|_{\mathcal{H}}\|F\|_{\mathcal{H}}+\|F\|_{\mathcal{H}}^2\right],
\end{equation}
where $c_{14}=c_{12}+c_{13}$.\\
Moving to the second term in \eqref{ee5}, we have
\begin{equation}\label{ee8}
2\Re\left(i\int_0^1 \left(\frac{\varphi}{\sigma}\right)^{\prime}f^1\la \bar{u}dx\right)=2\Re\left(i\int_0^1 \dfrac{\varphi^{\prime}}{\sigma} f^1\la \bar{u}dx\right)+2\Re\left(i\int_0^1 \dfrac{\varphi}{\sigma}\left(\dfrac{a^{\prime}-b}{a}\right) f^1\la \bar{u}dx\right).
\end{equation}
Using \eqref{lambdau}, the first term in the right hand side of \eqref{ee8} can be estimated as \begin{equation}\label{ee8,1}
\left|2\Re\left(i\int_0^1 \dfrac{\varphi^{\prime}}{\sigma} f^1\la \bar{u}dx\right)\right|\leq 2\|\varphi^{\prime}\|_{\infty}\ma{\int_0^1} \frac{1}{\sqrt{\sigma}}|f^1| \frac{1}{\sqrt{\sigma}}|\la \bar{u}|dx\leq c_{15} \left[\|U\|_{\mathcal{H}}\|F\|_{\mathcal{H}}+\|F\|_{\mathcal{H}}^2\right],
\end{equation}
where $\displaystyle c_{15}= 2\|\varphi^{\prime}\|_{\infty}c_0c_1$.\\
On the other hand, using the definition of $\varphi$ and Hypothesis \ref{hyp2}, the second term in the right hand side in  \eqref{ee8} can be estimated in the following way:
\begin{equation}\label{ee8,2}
\begin{array}{l}
\displaystyle
\left| 2\Re\left(i\int_0^1 \dfrac{ \varphi}{\sigma}\left(\dfrac{a^{\prime}-b}{a}\right) f^1\la \bar{u}dx\right)\right| 
\leq 2 \left(M_{0,2}+M_{1,2}+M_{2\epsilon}\right) \|f^1\| _{\frac{1}{\sigma}}\|\la u\|_{\frac{1}{\sigma}}\\
\displaystyle
\leq 2c_0 \left(1+\frac{K}{2}+M_{2\epsilon}\right)\|\sqrt{\eta} f^1_x\|\|\la u\|_{\frac{1}{\sigma}} 
\leq c_{16} \left[\|U\|_{\mathcal{H}}\|F\|_{\mathcal{H}}+\|F\|_{\mathcal{H}}^2\right],
\end{array}
\end{equation}
where $c_{16}=2c_0c_1 \left(1+\frac{K}{2}+M_{2\epsilon}\right)$, where $M_{2\epsilon}:= \left\|x\frac{a^{\prime}-b}{a}\right\|_{L^{\infty}({I_{2\epsilon}})}+\left\|(x-1)\frac{a^{\prime}-b}{a}\right\|_{L^{\infty}({I_{2\epsilon}})}$. \\
Then, by \eqref{ee8}, \eqref{ee8,1} and \eqref{ee8,2}  we obtain 
\begin{equation}\label{ee91}
\left|2\Re\left(i\int_0^1 \left(\frac{\varphi}{\sigma}\right)^{\prime}f^1\la \bar{u}dx\right)\right|\leq  c_{17} \left[\|U\|_{\mathcal{H}}\|F\|_{\mathcal{H}}+\|F\|_{\mathcal{H}}^2\right],
\end{equation}
where $c_{17}=c_{15}+c_{16}$,
and using  \eqref{ee7} and \eqref{ee91} in \eqref{ee5}, we can conclude that
\begin{equation}\label{E41}
\left|2\Re\left(i\int_0^1 \left(\frac{f^1\varphi}{\sigma}\right)_x\la \bar{u}dx\right)\right|\leq c_{18} \left[\|U\|_{\mathcal{H}}\|F\|_{\mathcal{H}}+\|F\|_{\mathcal{H}}^2\right]
\end{equation}
where $c_{18}=c_{14}+c_{17}$.\\

\underline{\textbf{Step 4.}}  The aim of this step is to show  the following two estimates
 \begin{equation}\label{last1}
\int_0^{1}\frac{1}{\sigma}|\la u|^2\leq \kappa_3 \left( \|U\|_{\mathcal{H}}^{3/2}\|F\|_{\mathcal{H}}^{1/2}+(1+\frac{1}{\la ^2})\left[\|U\|_{\mathcal{H}}\|F\|_{\mathcal{H}}+\|F\|_{\mathcal{H}}^2\right]\right)
\end{equation}
and 
\begin{equation}\label{last2}
\int_{0}^1\eta |u_x|^2 dx\leq \kappa_4\left( \|U\|_{\mathcal{H}}^{3/2}\|F\|_{\mathcal{H}}^{1/2}+(1+\frac{1}{\la ^2})\left[\|U\|_{\mathcal{H}}\|F\|_{\mathcal{H}}+\|F\|_{\mathcal{H}}^2\right]\right),
\end{equation}
where $\kappa_3=\frac{1}{1+\frac{K}{2}-(M_{0,2}+M_{1,2})}c_{23}$ and $\kappa_4=\frac{1}{1-\frac{K}{2}-(M_{0,1}+M_{1,1})}c_{23}$, for a suitable positive constant $c_{23}$ that will be determined below.\\
In order to prove the above estimates we start considering the first two terms on the left hand side of \eqref{e2},  and using the fact that $\varphi^{\prime}=\varphi_1+\varphi_2+\hat{\varphi}$ where $\hat{\varphi}=x\varphi_1^{\prime}+(x-1)\varphi_2^{\prime}$ with $\supp \hat{\varphi}=\overline{I_{2\epsilon}}$, we obtain
\begin{equation}\label{ee9}
\begin{array}{lr}
\displaystyle
\int_0^1 \frac{\varphi^{\prime}}{\sigma}|\la u|^2dx+\int_0^1 \eta \varphi^{\prime}|u_x|^2dx=\int_0^{x_1+2\epsilon} \frac{1}{\sigma}|\la u|^2dx+\int_{x_2-2\epsilon}^1 \frac{1}{\sigma}|\la u|^2dx+\int_{x_2-2\epsilon}^1\eta |u_x|^2dx\\
\displaystyle

+\int_0^{x_1+2\epsilon}\eta |u_x|^2+\int_{I_{2\epsilon}} \frac{\varphi_1+\varphi_2}{\sigma}|\la u|^2dx+\int_{I_{2\epsilon}}(\varphi_1+\varphi_2 )\eta |u_x|^2dx
\displaystyle
+\int_0^1 \frac{\hat{\varphi}}{\sigma}|\la u|^2dx+\int_0^1\hat{\varphi}\eta |u_x|^2dx.
\end{array}
\end{equation}
Estimating the last four terms  in the above equation, using  \eqref{stab1} and the fact that $\supp \hat{\varphi}=\overline{I}_{2\epsilon}$, we get
\begin{equation}\label{ee9,1}
\int_0^1 \frac{\hat{\varphi}}{\sigma}|\la u|^2dx\leq \|\hat{\varphi}\|_{L^{\infty}(I_{2\epsilon})}\int_{I_{2\epsilon}}\frac{1}{\sigma}|\la u|^2dx\leq c_{19} \left[\|U\|_{\mathcal{H}}\|F\|_{\mathcal{H}}+\|F\|_{\mathcal{H}}^2\right],
\end{equation}
\begin{equation}\label{ee9,2}
\int_0^1\hat{\varphi}\eta |u_x|^2dx \leq c_{20}  \left[\|U\|_{\mathcal{H}}\|F\|_{\mathcal{H}}+\|F\|_{\mathcal{H}}^2\right],
\end{equation}
and
\begin{equation}\label{ee9new}
\int_{I_{2\epsilon}} \frac{\varphi_1+\varphi_2}{\sigma}|\la u|^2dx+\int_{I_{2\epsilon}}(\varphi_1+\varphi_2 )\eta |u_x|^2dx\leq (2\kappa_1+2(\kappa_2^{\prime}+\frac{c_2}{\la^2})) \left[\|U\|_{\mathcal{H}}\|F\|_{\mathcal{H}}+\|F\|_{\mathcal{H}}^2\right],
\end{equation}
where $c_{19}= \|\hat{\varphi}\|_{L^{\infty}(I_{2\epsilon})}\kappa_1$ and $c_{20}=\|\hat{\varphi}\|_{L^{\infty}(I_{2\epsilon})}\left(\kappa_2^{\prime}+\dfrac{c_2}{\la ^2}\right)$.\\
Finally, using \eqref{E11}, \eqref{E21}, \eqref{E31}, \eqref{E41}, \eqref{ee9}, \eqref{ee9,1},  \eqref{ee9,2}, \eqref{ee9new}, \eqref{stab1} and Lemma \ref{Lemma01} in equation \eqref{e2}, we obtain
\begin{equation}
\begin{array}{ll}
\displaystyle
\int_0^{x_1+2\epsilon} \frac{1}{\sigma}|\la u|^2dx+\int_{x_2-2\epsilon}^1 \frac{1}{\sigma}|\la u|^2dx+\int_{x_2-2\epsilon}^1\eta |u_x|^2dx+\int_0^{x_1+2\epsilon}\eta |u_x|^2dx\\
\displaystyle
\leq (M_{0,2}+M_{1,2})\int_0^1\frac{1}{\sigma}|\la u|^2dx+(M_{0,1}+M_{1,1})\int_0^1\eta|u_x|^2dx+c_5 \|U\|_{\mathcal{H}}^{3/2}\|F\|_{\mathcal{H}}^{1/2}\\
\displaystyle
+ c_{21}\left[\|U\|_{\mathcal{H}}\|F\|_{\mathcal{H}}+\|F\|_{\mathcal{H}}^2\right]+

\frac{c_{22}}{\la ^2}\left[\|U\|_{\mathcal{H}}\|F\|_{\mathcal{H}}+\|F\|_{\mathcal{H}}^2\right],
\end{array}
\end{equation}
Then,  adding the above equation with the two equations in \eqref{stab1}, we get
\begin{equation}\label{MM00}
\begin{array}{ll}
\displaystyle
\int_0^{1} \frac{1}{\sigma}|\la u|^2dx+\int_0^{1}\eta |u_x|^2dx\leq (M_{0,2}+M_{1,2})\int_0^1\frac{1}{\sigma}|\la u|^2dx\\
\displaystyle
+(M_{0,1}+M_{1,1})\int_0^1\eta|u_x|^2dx+c_5 \|U\|_{\mathcal{H}}^{3/2}\|F\|_{\mathcal{H}}^{1/2}\\
\displaystyle
+ c_{21}\left[\|U\|_{\mathcal{H}}\|F\|_{\mathcal{H}}+\|F\|_{\mathcal{H}}^2\right]+

\frac{c_{22}}{\la ^2}\left[\|U\|_{\mathcal{H}}\|F\|_{\mathcal{H}}+\|F\|_{\mathcal{H}}^2\right],
\end{array}
\end{equation}
where  $c_{21}=c_5+c_6+c_{11}+c_{18}+c_{19}+\|\hat{\varphi}\|_{L^{\infty}(I_{2\epsilon})}\kappa_2^{\prime}+M_{1,2\epsilon}\kappa_1+M_{0,2\epsilon}\kappa_2^{\prime}+2\kappa_1+2\kappa_2^{\prime}$ ,  $c_{22}=\|\hat{\varphi}\|_{L^{\infty}}c_2+2c_2+M_{0,2\epsilon}c_2$ such that $M_{1,2\epsilon}:=\left\|x\frac{a^{\prime}-b}{a}\right\|_{L^{\infty}(I_{2\epsilon})}+\left\|(x-1)\frac{a^{\prime}-b}{a}\right\|_{L^{\infty}(I_{2\epsilon})}$ and $M_{0,2\epsilon}:=\left\|x\frac{b}{a}\right\|_{L^{\infty}(I_{2\epsilon})}+\left\|(x-1)\frac{b}{a}\right\|_{L^{\infty}(I_{2\epsilon})}$.\\
Now,  multiply \eqref{Eq1} by $\displaystyle\frac{K}{2\sigma}\overline{u}$, integrate over $(0,1)$, and take the real part, we get
\begin{equation}\label{NN0}
\begin{array}{ll}
\displaystyle
\frac{K}{2}\int_0^1\frac{1}{\sigma}|\la u|^2dx-\frac{K}{2}\int_0^1\eta|u_x|^2dx=-\frac{K}{2}\Re\left(\int_0^1 \frac{1}{\sigma}f^2\overline{u}dx\right)-\frac{K}{2}\Re\left(i\int_0^1 \frac{1}{\sigma}f^1\la\overline{u}dx\right)\\
\displaystyle
-\frac{K}{2}\Re\left(\int_0^1 \frac{1}{\sigma}h(x)f^1\overline{u}dx\right).
\end{array}
\end{equation}
For the estimation of the last terms in the above equation, by using Cauchy-Schwarz,  \eqref{HP} and \eqref{lambdau}, we obtain

\begin{equation}\label{NN1}
\begin{array}{ll}
\displaystyle \left| \frac{K}{2}\Re\left(\int_0^1 \frac{1}{\sigma}f^2\overline{u}dx\right)\right|\leq \frac{K}{2}c_0\left[\|U\|_{\mathcal{H}}\|F\|_{\mathcal{H}}+\|F\|_{\mathcal{H}}^2\right],
\end{array}
\end{equation}

\begin{equation}\label{NN2}
\begin{array}{ll}
\displaystyle\left|\frac{K}{2}\Re\left(i\int_0^1 \frac{1}{\sigma}f^1\la\overline{u}dx\right)\right|\leq \frac{K}{2}c_0c_1\left[\|U\|_{\mathcal{H}}\|F\|_{\mathcal{H}}+\|F\|_{\mathcal{H}}^2\right],
\end{array}
\end{equation}

\begin{equation}\label{NN3}
\begin{array}{ll}
\displaystyle \left|\frac{K}{2}\Re\left(\int_0^1 \frac{1}{\sigma}h(x)f^1\overline{u}dx\right)\right|\leq \frac{K}{2} \max\limits_{x\in [x_1,x_2]} h(x)c_0^2\left[\|U\|_{\mathcal{H}}\|F\|_{\mathcal{H}}+\|F\|_{\mathcal{H}}^2\right].
\end{array}
\end{equation}
So, using \eqref{NN1}-\eqref{NN3} in \eqref{NN0}, we get
\begin{equation}\label{NN00}
\frac{K}{2}\int_0^1\frac{1}{\sigma}|\la u|^2dx-\frac{K}{2}\int_0^1\eta|u_x|^2dx\leq \widetilde{c}_{20} \left[\|U\|_{\mathcal{H}}\|F\|_{\mathcal{H}}+\|F\|_{\mathcal{H}}^2\right],
\end{equation}
where $\widetilde{c}_{20}=\frac{K}{2}(c_0c_1+c_0+\max\limits_{x\in [x_1,x_2]} h(x)c_0^2)$.\\
Finally, adding \eqref{MM00} and \eqref{NN00}, we obtain
\begin{equation}
\begin{array}{ll}
\displaystyle
(1+\frac{K}{2}-(M_{0,2}+M_{1,2}))\int_0^{1} \frac{1}{\sigma}|\la u|^2dx+(1-\frac{K}{2}-(M_{0,1}+M_{1,1}))\int_0^{1}\eta |u_x|^2dx\leq \\
\displaystyle
+c_5 \|U\|_{\mathcal{H}}^{3/2}\|F\|_{\mathcal{H}}^{1/2}
+ (c_{21}+\kappa_1+\kappa_2^{\prime}+\widetilde{c}_{20})\left[\|U\|_{\mathcal{H}}\|F\|_{\mathcal{H}}+\|F\|_{\mathcal{H}}^2\right]+

\frac{c_{22}+c_2}{\la ^2}\left[\|U\|_{\mathcal{H}}\|F\|_{\mathcal{H}}+\|F\|_{\mathcal{H}}^2\right]
\\
\displaystyle
\leq c_{23} \left( \|U\|_{\mathcal{H}}^{3/2}\|F\|_{\mathcal{H}}^{1/2}+(1+\frac{1}{\la ^2})\left[\|U\|_{\mathcal{H}}\|F\|_{\mathcal{H}}+\|F\|_{\mathcal{H}}^2\right]\right),
\end{array}
\end{equation}
where $c_{23}=\max\{c_5, c_{21}+\kappa_1+\kappa_2^{\prime}+\widetilde{c}_{20},c_{22}+c_2\}$.
Finally, by Hypothesis \ref{hyp2}, we get the desired estimates \eqref{last1} and \eqref{last2}.\\
\underline{\textbf{Step 5.}}The goal of this step is to prove \eqref{prop2.9}. \\
Using the first equation in \eqref{eq1}, \eqref{HP} and the Young inequality, we get    
$$
\int_0^1\frac{1}{\sigma}\abs{v}^2dx\leq 2\int_0^1\frac{1}{\sigma}\abs{\la u}^2dx+2c_0^2\int_0^1\eta\abs{f^1_x}^2dx\leq 2\max\left\{1,c_0^2\right\}\left(\int_0^1\frac{1}{\sigma}\abs{\la u}^2dx+\int_0^1\eta \abs{f^1_x}^2dx\right).
$$  
Now, from the above estimate, \eqref{last1} and the fact that $\displaystyle{\int_0^1\eta \abs{f^1_x}^2dx}\leq \|F\|_{\mathcal{H}}^2\leq (1+\la^{-2})\|F\|^2_{\mathcal{H}}$, we obtain 
$$
\int_0^1\frac{1}{\sigma}\abs{v}^2dx\leq 2\max\left\{1,c_0^2\right\}(\kappa_3+1)\left( \|U\|_{\mathcal{H}}^{3/2}\|F\|_{\mathcal{H}}^{1/2}+\left(1+\frac{1}{\la ^2}\right)\left[\|U\|_{\mathcal{H}}\|F\|_{\mathcal{H}}+\|F\|_{\mathcal{H}}^2\right]\right);
$$
thus, thanks to this inequality and  again by \eqref{last1}, we get 
$$
\|U\|_{\mathcal{H}}^2\leq c_{24}\left( \|U\|_{\mathcal{H}}^{3/2}\|F\|_{\mathcal{H}}^{1/2}+\left(1+\frac{1}{\la ^2}\right)\left[\|U\|_{\mathcal{H}}\|F\|_{\mathcal{H}}+\|F\|_{\mathcal{H}}^2\right]\right),
$$
with $c_{24}=2\max\left\{1,c_0^2\right\}(\kappa_3+1)+\kappa_4$. 
Since $\displaystyle c_{24} \|U\|_{\mathcal{H}}^{3/2}\|F\|_{\mathcal{H}}^{1/2}\leq \frac{1}{2}\|U\|_{\mathcal{H}}^2+\frac{c_{24}^2}{2}\|U\|_{\mathcal{H}} \|F\|_{\mathcal{H}}$, by the previous inequality we have
\begin{equation}\label{eqstep5}
\frac{1}{2}\|U\|_{\mathcal{H}}^2\leq c_{25}\left(1+\frac{1}{\la^2}\right)\left[\|U\|_{\mathcal{H}}\|F\|_{\mathcal{H}}+\|F\|_{\mathcal{H}}^2\right],
\end{equation}
where $\displaystyle c_{25}= \left(\frac{c_{24}^2}{2}+c_{24}\right)$.
Now, since
\[\displaystyle c_{25}^2\left(1+\frac{1}{\la^2}\right)^2+c_{25}\left(1+\frac{1}{\la^2}\right) \leq \left(c_{25}+\frac{1}{2}+\frac{c_{25}}{\la^2}\right)^2\]
and
 \[\displaystyle c_{25}\left(1+\frac{1}{\la^2}\right)\|U\|_{\mathcal{H}}\|F\|_{\mathcal{H}}\leq \frac{1}{4}\|U\|_{\mathcal{H}}^2+ c_{25}^2\left(1+\frac{1}{\la^2}\right)^2\|F\|_{\mathcal{H}}^2,\]
 by \eqref{eqstep5} we have

\begin{equation*}
\|U\|_{\mathcal{H}}^2\leq 4\left(c_{25}+\frac{1}{2}+\frac{c_{25}}{\la^2}\right)^2 \|F\|_{\mathcal{H}}^2\leq 4\left(c_{25}+\frac{1}{2}\right)^2\left(1+\frac{1}{\la^2}\right)^2 \|F\|_{\mathcal{H}}^2.
\end{equation*}
Thus, $\displaystyle\|U\|_{\mathcal{H}}\leq \kappa \left(1+\frac{1}{\la^2}\right)  \|F\|_{\mathcal{H}}$, with $\kappa=2c_{25}+1$ and the proof of Proposition \ref{prop2} is completed.
\hfill$\square$

\noindent\textbf{Proof of Theorem \ref{Exponential Stab}.}
First, we will prove \eqref{H1}. Remark that it has been proved in Proposition \ref{prop1} that $0\in\rho(\mathcal{A})$.  Now, suppose \eqref{H1} is not true,  then there exists $\la_{\infty}\in \R^{\ast}$ such that $i\la_{\infty}\notin\rho(\mathcal{A})$. According to Remark A.1 in \cite{AkilCPAA},  page 25 in \cite{LiuZheng01},  and Remark A.3 in \cite{Akil2022},   then there exists $\left\{\left(\la_n,U_n=\left(u_n,v_n\right)\right)\right\}_{n\geq 1}\subset \R^{\ast}\times D\left(\mathcal{A}\right)$
with $\la_n\to\la_{\infty}$ as $n\to \infty$, $|\la_n|\leq|\la_\infty|$ and$ \|U_n\|_{\mathcal{H}}=1$ such that $\left(i\la_nI-\mathcal{A}\right)U_n=F_n\to 0\quad \text{in}\quad \mathcal{H}\quad\text{as}\quad n\to 0.$
By taking $F=F_n$, $U=U_n$ and $\la =\la_n$ in Proposition \ref{prop2} such that $ \|F_n\|_{\mathcal{H}}\to 0$  as $n\to \infty$,  we get $ \|U_n\|_{\mathcal{H}}\to 0$, which contradicts that  $\|U_n\|_{\mathcal{H}}=1$. Thus, condition \eqref{H1} is true. Next, we will prove \eqref{H2} by a contradiction argument. Suppose there exists 
\begin{equation*}
\left\{\left(\la_n,U_n=\left(u_n,v_n\right)\right)\right\}_{n\geq 1}\subset \R^{\ast}\times D\left(\mathcal{A}\right)
\end{equation*}
with $|\la _n|>1$ without affecting the result, such that $|\la _n|\to\infty$, $\|U_n\|_{\mathcal{H}}=1$ and there exists a sequence $F_n=(f^1_n,f_n^2)\in\mathcal{H}$ such that
$$\left(i\la_nI-\mathcal{A}\right)U_n=F_n\to 0\quad \text{in}\quad \mathcal{H}\quad\text{as}\quad n\to 0.$$
From Proposition \ref{prop2}, as  $|\la _n|\to\infty$, we get$\|U_n\|_{\mathcal{H}}=o(1)$, which contradicts $\|U_n\|_{\mathcal{H}}=1$. Thus, condition \eqref{H2} holds true. The result follows from Huang-Pr\"uss Theorem (see  \cite{Huang01} and  \cite{pruss01}) and the proof is completed. \hfill$\square$

\section{Stabilization of  connected degenerate and non-degenerate wave equations  with drift and a single boundary damping}
\noindent This section is devoted to study the well-posedness and the exponential stability of \eqref{System2}. 
\subsection{Well-Posedness}
In this subsection, we will study the well-posendness of  \eqref{System2}.  To this aim, using the definition of $\sigma$ given in \eqref{sigma}, we rewrite the system \eqref{System2} as follows:
\begin{equation}\label{S2}
\left\{\begin{array}{lll}
u^1_{tt}-\sigma(\eta u^1_{x})_x=0,& (x,t) \in (0,1)\times\R^+_{\ast},\\[0.1in]
y^1_{tt}-y^1_{xx}=0, &(x,t) \in (1,L)\times\R^+_{\ast},\\[0.1in]
u^1(0,t)=0, u^1(1,t)=y^1(1,t), &t\in \R^+_{\ast},\\[0.1in]
(\eta u^1_x)(1,t)=y^1_x(1,t), y^1_x(L,t)=-y^1_t(L,t), &t\in \R^+_{\ast},\\[0.1in]
u^1(x,0)=u^1_0(x), u^1_t(x,0)=u^1_1(x),&x\in (0,1),\\[0.1in]
y^1(x,0)=y^1_0(x), y^1_t(x,0)=y^1_1(x),&x\in (1,L).
\end{array}\right.
\end{equation}
By using the same arguments as in Section \ref{sec2}, we introduce the following Hilbert spaces 
$$H_{\frac{1}{\sigma},\ell}^1(0,1):=\left\lbrace u^1\in H_{\frac{1}{\sigma}}^1(0,1); u^1(0)=0\right\rbrace\quad\text{and}\quad H_{\frac{1}{\sigma},\ell}^2(0,1):=\left\lbrace u^1\in  H_{\frac{1}{\sigma},\ell}^1(0,1); Bu^1\in L_{\frac{1}{\sigma}}^2(0,1) \right\rbrace$$
equipped with the following norm
$$\|u^1\|_{\frac{1}{\sigma},\ell}^2=\|u^1\|_{\frac{1}{\sigma},1}^2=\|u\|^2_\frac{1}{\sigma}+\int_0^1 \eta |u_x|^2dx\quad\text{and}\quad  \|u^1\|_{\frac{1}{\sigma},2}^2=\|u\|^2_{\frac{1}{\sigma},\ell}+\int_0^1 \sigma |(\eta u_x)_x|^2dx .$$
Thanks to \eqref{HP}, we obtain that $\|u\|_{\frac{1}{\sigma},\ell}^2$ is equivalent to $\displaystyle\int_0^1 \eta |u_x|^2dx$. 
Let $\ibt{(u^1,u^1_t, y^1, y^1_t)} $ be a regular solution of the system \eqref{S2}. The energy of the system is given by
\begin{equation}\label{Energy2}
E_1(t)=\frac{1}{2}\int_0^{1}\left(\dfrac{1}{\sigma}|u^1_t|^2+\eta|u^1_x|^2\right)dx+\frac{1}{2}\int_{1}^L\left(|y^1_t|^2+|y^1_x|^2\right)dx
\end{equation}
and we obtain that
$$
\dfrac{d}{dt} E_1(t)=-|z^1(L)|^2\leq 0.
$$
Thus, the system \eqref{S2} is dissipative in the sense that its energy is a non increasing function with respect to the time variable $t$. 
We define the energy Hilbert space $\mathcal{H}_1$ by 
$$
\mathcal{H}_1=\left\{\mathcal{U}=(u,^1 v,^1 y,^1 z^1)\in H^1_{\frac{1}{\sigma},\ell}(0,1)\times L^2_{\frac{1}{\sigma}}(0,1)\times H^1(1,L)\times L^2(1,L); u^1(1)=y^1(1) \right\}
$$
equipped with the following inner product
$$
\left<\mathcal{U},\widetilde{\mathcal{U}}\right>_{\mathcal{H}_1}=\int_0^{1}\left(\eta u^1_x{\overline{\widetilde{u^1_x}}}+\frac{1}{\sigma} v^1\overline{\widetilde{v^1}}\right)dx+\int_{1}^L\left(y^1_x{\overline{\widetilde{y^1_x}}}+z^1\overline{\widetilde{z^1}}\right)dx ,
$$
and endowed with the the norm 
\begin{equation}\label{NormEq}
\|\mathcal{U}\|^2_{\mathcal{H}_1}=\int_0^{1} \left(\eta|u^1_x|^2+\frac{1}{\sigma}|v^1|^2\right)dx+\int_{1}^L \left(|y^1_x|^2+|z^1|^2\right) dx
\end{equation}
for all
 $U=(u^1,v^1, y^1, z^1)^{\top}$ and $\widetilde{\mathcal{U}}=(\widetilde{u^1},\widetilde{v^1}, \widetilde{y^1}, \widetilde{z^1})^{\top}\in\mathcal{H}_1$ .
Noting that the standard norm on $\mathcal{H}_1$ is 
\begin{equation}\label{Norm}
\|\mathcal{U}\|^2_{\mathcal{S}}=\|v^1\|^2_{\frac{1}{\sigma}}+\|u^1\|^2_{\frac{1}{\sigma},\ell}+\|z^1\|_{L^2(1,L)}^2+\|y^1\|_{L^2(1,L)}+\|y^1_x\|_{L^2(1,L)}^2.
\end{equation}

\begin{Lemma}\label{Norms}
The two norms  $\|\cdot\|_{\mathcal{H}_1}$ and $\|\cdot\|_{\mathcal{S}}$ are equivalent in $\mathcal{H}_1$, i.e.  there exist two positive constants $C_1, C_2$, independent of $\mathcal{U}$, such that
\begin{equation}\label{equiv}
C_1\|\mathcal{U}\|^2_{\mathcal{S}}\leq \|\mathcal{U}\|^2_{\mathcal{H}_1}\leq C_2 \|\mathcal{U}\|^2_{\mathcal{S}}
\end{equation}
for all $\mathcal{U}=(u^1,v^1,y^1,z^1)\in \mathcal{H}_1$.

\begin{proof}
The inequality on the right hand side is evident with $C_2=1$. \\
Thanks to the boundary and the transmission conditions ($u^1(0)=0$ and $u^1(1)=y^1(1)$), we have $\displaystyle u^1(1)=\int_0^{1}u^1_xdx$ and
$$y^1(x)=u^1(1)+\int_{1}^xy^1_s(s)ds.$$\\
Moreover, by the Young and the Cauchy-Schwarz inequalities we have
\begin{equation*}\label{young1}
|u^1(1)|^2\leq \max_{x\in[0,1]}\frac{1}{\eta} \int_0^{1} \eta |u^1_x|^2dx\quad \text{and}\quad |y^1(x)|^2\leq 2|u^1(1)|^2+2(L-1)\|y_x^1\|_{L^2(1,L)}^2;
\end{equation*}
thus
\begin{equation}\label{young3}
\int^L_{1}|y^1(x)|^2dx\leq  R\left(\int_0^{1} \eta |u^1_x|^2dx+\int_{1}^L|y^1_x|^2dx\right),
\end{equation}
where $R:=2(L-1)\max\left\{\max_{x\in[0,1]}\frac{1}{\eta},L-1\right\}.$
Hence, by  \eqref{young3}, 
\begin{equation*}
\|\mathcal{U}\|^2_{\mathcal{S}}\leq (1+R) \left( \|v^1\|^2_{\frac{1}{\sigma}}+\|u^1\|^2_{\frac{1}{\sigma},\ell}+\|z^1\|_{L^2(1,L)}^2+\|y^1_x\|_{L^2(1,L)}\right)
\end{equation*}
and the left hand side of inequality of \eqref{equiv} is true with
$\displaystyle C_1=\frac{1}{1+R}$.
\end{proof}
\end{Lemma}
\noindent Finally, we define the unbounded linear operator $\mathcal{A}_1$ by 
$$
\mathcal A \mathcal U=\mathcal{A}_1(u^1,v^1, y^1, z^1)^{\top}=\left(v^1,\sigma(\eta  u^1_x)_x,z^1,y^1_{xx}\right)^{\top}.
$$
for all $\mathcal{U}=(u^1,v,^1 y^1, z^1)^{\top}\in D(\mathcal{A}_1)$, where
\begin{equation*}
D(\mathcal{A}_1)=\left\{\begin{array}{c}
\displaystyle{\mathcal{U}=(u^1,v^1,y^1,z^1)\in \mathcal{H}_1;\ (v^1,z^1)\in H_{\frac{1}{\sigma},\ell}^1(0,1)\times H^1(1,L),\ u^1\in H_{\frac{1}{\sigma},\ell}^2(0,1)},\\[0.1in]  
y^1\in  H^2(L,1),\ (\eta u^1_x)_x(1)=y^1_x(1), \ y^1_x(L)=-z^1(L)
\end{array}
\right\},
\end{equation*}
Hence, we can rewrite \eqref{S2} as the following Cauchy problem
\begin{equation}\label{Cauchy1}
\mathcal{U}_t=\mathcal{A}_1\mathcal{U},\quad\mathcal{ U(}0)=\mathcal{U}_0\quad \text{where}\quad\mathcal{U}_0=\left(u^1_0,u^1_1, y_0^1, y_1^1\right)^{\top}.
\end{equation}
We have that $\Re(\mathcal{A}_1\mathcal{U}, \mathcal{U})_{\mathcal{H}_1}=-|z^1(L)|^2$,
which implies that $\mathcal{A}_1$ is dissipative.  Now,  let $\mathcal{G}=(g_1,g_2, g_3, g_4)\in \mathcal{H}_1 $; by Lax-Milgram Theorem, one can prove the existence of $\mathcal{U}\in D(\mathcal{A}_1)$ such that
$$-\mathcal{A}_1\mathcal{U}=\mathcal{G.}$$
Therefore, the unbounded linear operator  $\mathcal{A}_1$ is m-dissipative in $\mathcal{H}_1, $ thus $0\in\rho(\mathcal{A}_1)$ and, by Lumer-Phillips Theorem (see, e.g., \cite{LiuZheng01} and \cite{Pazy01}), $\mathcal{A}_1$ generates a $C_0$- semigroup of contractions $(\mathcal T_1(t))_{t\ge0}=(e^{t\mathcal{A}_1})_{t\geq 0}$ in $\mathcal{H}_1$.  Hence, the solution of the Cauchy problem  \eqref{Cauchy1} admits the following representation 
$$\mathcal{U}(t)=e^{t\mathcal{A}_1}\mathcal{U}_0\quad t\geq 0,$$
which leads to the well-posedness of \eqref{Cauchy1}. The following result is immediate.

\begin{theoreme}
	For any $U_0\in\mathcal{H}_1$, problem \eqref{Cauchy1} admits a unique weak solution satisfying 
	$$
	\mathcal{U}(t)\in C^0(\R^+;\mathcal{H}_1). 
	$$
	Moreover, if  $\mathcal{U}_0\in D(\mathcal{A}_1) $, \eqref{Cauchy1} admits a unique strong solution $\mathcal{U}$ satisfying
	$$
	\mathcal{U}(t)\in C^1(\R^+,\mathcal{H}_1)\cap C^0(\R^+,D(\mathcal{A}_1)).
	$$
\end{theoreme}
\vspace{0.4cm}

\subsection{Exponential Stability}
This subsection focuses on examining the exponential stability of system \eqref{S2}. We begin introducing a main hypothesis that will serve as the basis for studying the system's exponential stability.  Denote by 
\begin{equation}
N_1:=\left\|x\frac{a^{\prime}-b}{a}\right\|_{L^{\infty}(0,1)}\quad\text{and}\quad
N_2:=\left\|x\frac{b}{a}\right\|_{L^{\infty}(0,1)}.
\end{equation}

\begin{Hypothesis}\label{HypothesisExp2}
Assume Hypothesis \ref{hyp1}, $a$  \eqref{WD} or \eqref{SD} and the functions $a$ and $b$ such that 
\center{  $N_1<1+\frac{K}{2}$ and $ N_2<1-\frac{K}{2}$.}
\end{Hypothesis}

\begin{ex}\normalfont
$(1)$ Example for the (WD) case: $a(x)=\sqrt{x}$ and $b(x)=\frac{1}{2}$.  In this case $K=\frac{1}{2}$ and it is easy to see that $N_1=N_2=\frac{1}{2}$. So, Hypothesis \ref{HypothesisExp2} is satisfied. \\
$(2)$ Example for the (SD) case: $a(x)=x\sqrt{x}$ and $b(x)=\frac{x}{8}$. In this case $K=\frac{3}{2}$ and it is easy to see that $N_1=\frac{3}{2}$ and $N_2=\frac{1}{8}$. So, Hypothesis \ref{HypothesisExp2} is satisfied. 

\end{ex}
\begin{Theorem}\label{Exponential Stab2}
Assume Hypothesis \ref{HypothesisExp2}. Then, the  $C_0-$semigroup of contractions $\left(\mathcal T_1(t)\right)_{t\geq 0}$ is exponentially  stable, i.e. there exist constants $M\geq 1$ and $\tau>0$, independent of $\mathcal{U}_0$, such that 
$$
\left\|e^{t\mathcal{A}_1}\mathcal{U}_0\right\|_{\mathcal{H}_1}\leq Me^{-\tau t}\|\mathcal{U}_0\|_{\mathcal{H}_1},\qquad t\geq 0.
$$
\end{Theorem}
\noindent According to Huang \cite{Huang01} and Pruss \cite{pruss01}, we have to check if the following conditions hold:
\begin{equation}\label{R1}\tag{${\rm R1}$}
 i\mathbb{R}\subseteq \rho\left(\mathcal{A}_1\right)
\end{equation}
and 
\begin{equation}\label{R2}\tag{${\rm R2}$}
\displaystyle{\sup_{\la\in \R}\|\left(i\la I-\mathcal{A}_1\right)^{-1}\|_{\mathcal{L}\left(\mathcal{H}_1\right)}=O(1).}
\end{equation}
The following proposition is a technical finding that will be used to prove Theorem \ref{Exponential Stab2}.
\begin{Proposition}\label{prop2'}
Assume Hypothesis \ref{HypothesisExp2} and let $(\lambda,\mathcal{U}:=(u,v, y, z))\in \R^{\ast}\times D(\mathcal{A}_1)$, with $\ma{\lambda\neq 0}$, such that
\begin{equation}\label{Q0}
\left(i\la I-\mathcal{A}_1\right)\mathcal{U}=G:=(g^1,g^2, g^3, g^4)\in \mathcal{H}_1 ,
\end{equation}
i.e.
\begin{eqnarray}
i\la u^1-v^1 &=&g^1,\label{Q1}\\
i\la v^1-\sigma(\eta u^1_x)_x&=&g^2,\label{Q2}\\
i\la y^1-z^1 &=&g^3,\label{Q3}\\
i\la z^1-y^1_{xx}&=&g^4,\label{Q4}.
\end{eqnarray}
Then the following inequality holds:
\begin{equation}\label{mainresult}
\|\mathcal{U}\|_{\mathcal{H}_1}\leq  m \left(1+\frac{1}{\la^2}\right) \|G\|_{\mathcal{H}_1},
\end{equation}
where $m$ is a suitable positive constant independent of $\lambda$.
\end{Proposition}
Before proving the above proposition, observe that,
by the dissipation of the energy, we have
\begin{equation}
|z^1(L)|^2= |\Re\left<\mathcal{A}\mathcal{U},\mathcal{U}\right>_{\mathcal{H}}|=|\Re\left<i\la I-\mathcal{A}\mathcal{U},\mathcal{U}\right>_{\mathcal{H}_1}|=|\left<G,\mathcal{U}\right>_{\mathcal{H}_1}|\leq \ma{\|\mathcal{U}\|_{\mathcal{H}_1}\|G\|_{\mathcal{H}_1}},
\end{equation}
and, using the equality $y^1_x(L)=-z^1(L)$, 
\begin{equation}\label{y1xL}
\ma{\abs{y^1_x(L)}^2\leq  \|\mathcal{U}\|_{\mathcal{H}_1}\|G\|_{\mathcal{H}_1}\leq \|\mathcal{U}\|_{\mathcal{H}_1}\|G\|_{\mathcal{H}_1}+\|G\|_{\mathcal{H}_1}^2.}
\end{equation}

\noindent Moreover, thanks to the transmission condition  $g^3(1)=g^1(1)$, one has
\begin{equation}\label{g1L1}
|g^3(1)|=|g^1(1)|=\left|\int_0^{1} g_x^1dx\right|\leq \sqrt{\max\limits_{x\in[0,1]}\frac{1}{\eta}}\|\ma{\sqrt{\eta}}g^1_x\|\leq m_1 \|G\|_{\mathcal{H}_1}
\end{equation}
where $m_1:=\sqrt{ \max\limits_{x\in[0,1]}\frac{1}{\eta}}. $
Hence,
 \begin{equation}\label{g3L1}
|g^3(1)|\leq m_1 \|G\|_{\mathcal{H}_1}
\end{equation}
and using the fact that $\displaystyle g^3(L)=\int_{1}^L g_x^3 dx+g^3(1)$, we can deduce
\begin{equation}\label{g3L}
 |g^3(L)|^2\leq 2(L-1)\|g^3_x\|^2+2m_1^2 \|G\|_{\mathcal{H}_1}^2\leq 2(L-1+m_1^2)\|G\|_{\mathcal{H}_1}^2\leq m_2\left[\|\mathcal{U}\|_{\mathcal{H}_1}\|G\|_{\mathcal{H}_1}+\|G\|_{\mathcal{H}_1}^2\right]
\end{equation}
where $m_2=2(L-1+m_1^2)$.
Now, since $ i\la y^1 -z^1 =g^3$, we get
\begin{equation}\label{lambdayL}
|\la y^1 (L)|^2\leq 2|z^1 (L)|^2+2|g^3(L)|^2\leq 2\|\mathcal{U}\|_{\mathcal{H}_1}\|G\|_{\mathcal{H}_1}+2m_2\|G\|_{\mathcal{H}_1}^2\leq m_3\left[\|\mathcal{U}\|_{\mathcal{H}_1}\|G\|_{\mathcal{H}_1}+\|G\|_{\mathcal{H}_1}^2\right] 
\end{equation}
where $m_3=2\max\{1, m_2\}$ and, using \eqref{Q1}, the Young inequality and \eqref{HP}, we deduce
\begin{equation}\label{lambdau1}
\begin{array}{ll}
\displaystyle
\|\la u^1\|_{\frac{1}{\sigma}}\leq \|v\|_{\frac{1}{\sigma}}+\sqrt{C_{HP}}\|g^1_x\|
\leq m_0\left(\|v^1\|_{\frac{1}{\sigma}}+\|\sqrt{\eta}g_x^1\|\right)\leq m_0 \left[\|\mathcal{U}\|_{\mathcal{H}_1}+\|G\|_{\mathcal{H}_1}\right]
\end{array}
\end{equation}
where $\displaystyle m_0= \max\left(1, \widetilde{m}_0\right)$ with $\displaystyle\widetilde{m}_0=\sqrt{C_{HP}\max_{x\in[0,1]}\frac{1}{\eta}}$.\\
Using the equivalence between the norms given in Lemma \ref{Norms} and the fact that $G\in \mathcal{H}_1$, we obtain 
\begin{equation}\label{g3}
\|g^3\|_{L^2(1,L)}\leq\|g^3\|_{H^1(1,L)} \leq \|G\|_{\mathcal{S}}\leq \ma{\sqrt{1+R}} \|G\|_{\mathcal{H}_1}
\end{equation}
\ma{where $R$ is defined in \eqref{young3}}; thus by \eqref{g3} and the equality $i\la y^1=z^1+g^3$, we obtain
\begin{equation}\label{lambday}
\|\la y^1\|_{L^2(1,L)}\leq \|z^1\|_{L^2(1,L)}+\|g^3\|_{L^2(1,L)}\leq \widetilde{m}_1 [ \|\mathcal{U}\|_{\mathcal{H}_1}+ \|G\|_{\mathcal{H}_1}],
\end{equation}
where $\ma{\widetilde{m}_1=\ibt{\sqrt{1+R}}}$. 
Substituting $v^1 =i\la u^1 -g^1$  and  $z^1 =i\la y^1 -g^3$ in \eqref{Q2} and \eqref{Q4} respectively we get,
\begin{equation}\label{Q56}
\la ^2u^1+\sigma(\eta u^1_x)_x=-(g^2+i\la g^1)\quad \text{and}\quad \la ^2y^1+y^1_{xx}=-(g^4+i\la g^3).
\end{equation}
Now, we are able to prove Proposition \ref{prop2'}.\\
\textbf{Proof of Proposition \ref{prop2'}.}We divide the proof into several steps.
\\
\underline{\textbf{Step 1.}}  The aim of this step is to show that  the solution $U = (u^1, v^1, y^1, z^1)\in D(\mathcal{A}_1)$ of \eqref{Q0} satisfies the following estimates
\begin{equation}\label{est5}
\int_{1}^{L}\left(|\la y^1|^2+|y^1 _x|^2\right)dx\leq m_5 \left[\|\mathcal{U}\|_{\mathcal{H}_1}\|G\|_{\mathcal{H}_1}+\|G\|_{\mathcal{H}_1}^2\right]
\end{equation}
and 
\begin{equation}\label{FinalEq4}
\int_{1}^L| z^1|^2dx\leq 2(m_5+\ibt{R+1})\left[\|\mathcal{U}\|_{\mathcal{H}_1}\|G\|_{\mathcal{H}_1}+\|G\|_{\mathcal{H}_1}^2\right],
\end{equation}
\ma{where $m_5$ will be determined below}. Let $\alpha, \beta\in \R$ and define the function $\displaystyle g_{(\alpha,\beta)}(x):=(\alpha+\beta)x-(\alpha +\beta L)$.\\
As a first step, multiply the second equation in \eqref{Q56} by $-2g_{(\alpha,\beta)}(x)\overline{y^1_x}$ and  integrate by parts over $(1,L)$. Then, taking the real part, we have
\begin{equation}\label{G1}
\begin{array}{ll}
\displaystyle
\int_{1}^{L}g_{(\alpha,\beta)}^{\prime}(x)\left(|\la y^1 |^2+|y^1_x|^2\right)dx-g_{(\alpha,\beta)}(L)\left(|\la y^1 (L)|^2+|y^1_x(L)|\right)\\
\displaystyle
+g_{(\alpha,\beta)}(1)\left(|\la y^1 (1)|^2+|y^1_x(1)^2|\right)=2\Re\left(\int_{1}^{L}g^4 g_{(\alpha,\beta)}(x)\overline{y^1_x}dx\right)-2\Re\left(i\la\int_{1}^{L}g^3_x g_{(\alpha,\beta)}(x)\overline{y^1}dx\right)\\ 
\displaystyle
-2\Re\left(i\la\int_{1}^{L}g^3 g^{\prime}_{(\alpha,\beta)}(x)\overline{y^1}dx\right)+2\Re\left[i\la g^3(L) g_{(\alpha,\beta)}(L)\overline{y^1}(L)\right]-2\Re\left[i\la g^3(1) g_{(\alpha,\beta)}(1)\overline{y^1}(1)\right].
\end{array}
\end{equation}
Taking $\beta=0$ and $\alpha=1$ in \eqref{G1}, we get
\begin{equation}\label{G2}
\begin{array}{ll}
\displaystyle
\int_{1}^{L}\left(|\la y^1 |^2+|y^1 _x|^2\right)dx=(L-1)|\la y^1(L)|^2+(L-1)|y^1 _x(L)|^2+2\Re\left(\int_{1}^{L}g^4 g_{(1,0)}(x)\overline{y^1_x}dx\right)\\ 
\displaystyle-2\Re\left(i\la\int_{1}^{L}g^3_x g_{(1,0)}(x)\overline{y^1}dx\right)
-2\Re\left(i\la\int_{1}^{L}g^3\overline{y^1}dx\right)+2\Re\left[i\la g^3(L) g_{(1,0)}(L)\overline{y^1}(L)\right].
\end{array}
\end{equation}
Now, we will  estimate the terms on the right hand side of the above equation.  To this aim, by \eqref{g3L},  \eqref{lambdayL},  \eqref{g3} and \eqref{lambday}, we have
\begin{equation}\label{est1}
\left|2\Re\left(\int_{1}^{L}g^4 g_{(1,0)}(x)\overline{y^1_x}dx\right)\right|\leq 2(L-1)\|\mathcal{U}\|_{\mathcal{H}_1}\|G\|_{\mathcal{H}_1},
\end{equation}
\begin{equation}\label{est2}
\left|2\Re\left(i\la\int_{1}^{L}g^3_x g_{(1,0)}(x)\overline{y^1}dx\right)\right|\leq 2(L-1)\|G\|_{\mathcal{H}_1}\|\la y^1 \|\leq 2\widetilde{m}_1(L-1)\left[\|\mathcal{U}\|_{\mathcal{H}_1}\|G\|_{\mathcal{H}_1}+\|G\|_{\mathcal{H}_1}^2\right],
\end{equation}
\begin{equation}\label{est3}
\left|2\Re\left(i\la\int_{1}^{L}g^3 \overline{y^1}dx\right)\right|\leq 2\|g^3 \| \|\la y^1 \|\leq 2\sqrt{1+R}\,\widetilde{m}_1 \left[\|\mathcal{U}\|_{\mathcal{H}_1}\|G\|_{\mathcal{H}_1}+\|G\|_{\mathcal{H}_1}^2\right], 
\end{equation}
\begin{equation}\label{est4}
\begin{array}{l}
\left|2\Re\left[i\la g^3(L) g_{(1,0)}(L)\overline{y^1}(L)\right]\right|\leq \ma{(L-1)^2}|\la y^1 (L)|^2+|g^3(L)|^2
\leq  m_4 \left[\|\mathcal{U}\|_{\mathcal{H}_1}\|G\|_{\mathcal{H}_1}+\|G\|_{\mathcal{H}_1}^2\right],
\end{array}
\end{equation}
where $m_4=\ma{(L-1)^2m_3+m_2}$. Thus, \ma{from the above estimations,  \eqref{y1xL} and \eqref{lambdayL} we get the following one}
\begin{equation*}
\int_{1}^{L}\left(|\la y^1 |^2+|y^1 _x|^2\right)dx\leq m_5 \left[\|\mathcal{U}\|_{\mathcal{H}_1}\|G\|_{\mathcal{H}_1}+\|G\|_{\mathcal{H}_1}^2\right],
\end{equation*}
where $\ma{m_5=(L-1)\left[Lm_3+3+2\widetilde{m}_1\right]+2\sqrt{1+R}\,\widetilde{m}_1+m_2}$.\\
From the above inequality, \eqref{Q3} \ma{and the fact that $\displaystyle{\int_{1}^L| g^3|^2dx\leq \|G\|^2_{\mathcal{S}}\leq \ibt{(1+R)}\|G\|_{\mathcal{H}_1}^2}$} we obtain
\begin{equation*}
\int_{1}^L| z^1|^2dx\leq 2\int_{1}^L|\la y^1|^2dx+2\int_{1}^L| g^3|^2dx\leq \ibt{2(m_5+R+1)}\left[\|\mathcal{U}\|_{\mathcal{H}_1}\|G\|_{\mathcal{H}_1}+\|G\|_{\mathcal{H}_1}^2\right].
\end{equation*}
\underline{\textbf{Step 2.}}  The aim of this step is to show that  the solution $U = (u^1, v^1, y^1, z^1)\in D(\mathcal{A}_1)$ of \eqref{Q0} satisfies the following estimates
 \begin{equation}\label{BDu}
|\la u^1 (1)|^2\leq m_7  \left[\|\mathcal{U}\|_{\mathcal{H}_1}\|G\|_{\mathcal{H}_1}+\|G\|_{\mathcal{H}_1}^2\right]\quad\text{and}\quad|\ibt{ (\eta u^1_x)(1)|}^2\leq m_8  \left[\|\mathcal{U}\|_{\mathcal{H}_1}\|G\|_{\mathcal{H}_1}+\|G\|_{\mathcal{H}_1}^2\right],
\end{equation}
where $m_7$ and $m_8$ will be determined below.\\
Taking $\alpha=0$ and $\beta=1$ in \eqref{G1}, we get
\begin{equation}\label{est6}
\begin{array}{ll}
\displaystyle
(L-1)|\la y^1 (1)|^2+(L-1)|y_x(1)|^2=\int_{1}^{L}\left(|\la y^1 |^2+|y^1_x|^2\right)dx-2\Re\left(\int_{1}^{L}g^4 g_{(0,1)}(x)\overline{y^1}_xdx\right)\\ 
\displaystyle+2\Re\left(i\la\int_{1}^{L}g^3_x g_{(0,1)}(x)\overline{y^1}dx\right)
+2\Re\left(i\la\int_{1}^{L}g^3\overline{y^1}dx\right)+2\Re\left[i\la g^3(1) g_{(0,1)}(1)\overline{y^1}(1)\right].
\end{array}
\end{equation}
To estimate the last term in the right hand side of the above equation, we use the Young inequality and \ma{\eqref{g3L1}}, obtaining
\begin{equation*}\label{est7}
\begin{array}{l}
\ma{\displaystyle 
\left|2\Re\left[i\la g^3(1) g_{(0,1)}(1)\overline{y^1}(1)\right]\right|\leq (L-1)\left(\frac{1}{2}|\la y^1(1)|^2+2|g^3(1)|^2\right)\leq \frac{1}{2}(L-1)|\la y^1(1) |^2+}\\
\displaystyle 
\ma{2(L-1)m_1^2\|G\|^2_{\mathcal{H}_1}\leq \frac{1}{2}(L-1)|\la y^1(1) |^2+2(L-1)m_1^2\left[\|\mathcal{U}\|_{\mathcal{H}_1}\|G\|_{\mathcal{H}_1}+\|G\|_{\mathcal{H}_1}^2\right]}.
\end{array}
\end{equation*}
Now, using the above inequality, observing that estimates \eqref{est1},\eqref{est2}, and \eqref{est3} are true also for $g_{(0,1)}(x)$ and using \eqref{est5}  in \eqref{est6}, we obtain
\begin{equation}
\frac{1}{2}(L-1)|\la y^1 (1)|^2+(L-1)|y^1 _x(1)|^2\leq m_6 \left[\|\mathcal{U}\|_{\mathcal{H}_1}\|G\|_{\mathcal{H}_1}+\|G\|_{\mathcal{H}_1}^2\right],
\end{equation}
where $m_6=\ibt{m_5+2\sqrt{1+R}\widetilde{m_1}+2(L-1)(1+\widetilde{m_1}+m_1^2)}$.
Therefore, 
\begin{equation}\label{BDy}
|\la y^1 (1)|^2\leq m_7  \left[\|\mathcal{U}\|_{\mathcal{H}_1}\|G\|_{\mathcal{H}_1}+\|G\|_{\mathcal{H}_1}^2\right]\quad\text{and}\quad|y^1 _x(1)|^2\leq m_8  \left[\|\mathcal{U}\|_{\mathcal{H}_1}\|G\|_{\mathcal{H}_1}+\|G\|_{\mathcal{H}_1}^2\right],
\end{equation}
with $\displaystyle m_7=\frac{2}{L-1}m_6$ and $\displaystyle m_8=\frac{1}{L-1}m_6$.
Now, using the  transmission conditions $\ma{y^1(1)=u^1(1)}$ and $\ma{y^1_x(1)=(\eta u^1_x)(1)}$, we obtain \eqref{BDu}.\\
\underline{\textbf{Step 3.}}  The aim of this step is to show that  the solution $U = (u^1, v^1, y^1, z^1)\in D(\mathcal{A}_1)$ of \eqref{Q0} satisfies the following estimates
\begin{equation}\label{FinalEq1}
\int_0^{1}\frac{1}{\sigma}|\la u^1|^2dx\leq m_{17} \left(1+\frac{1}{\la^2}\right)  \left[\|\mathcal{U}\|_{\mathcal{H}_1}\|G\|_{\mathcal{H}_1}+\|G\|_{\mathcal{H}_1}^2\right]
\end{equation}
and
\begin{equation}\label{FinalEq2}
\int_0^{1}\eta|u^1_x|^2dx\leq m_{18} \left(1+\frac{1}{\la^2}\right)  \left[\|\mathcal{U}\|_{\mathcal{H}_1}\|G\|_{\mathcal{H}_1}+\|G\|_{\mathcal{H}_1}^2\right],
\end{equation}
where $m_{17}$ and $m_{18}$ are positive constants to be determined below.\\
Multiplying the first equation in \eqref{Q56}  by $\displaystyle\frac{K}{2\sigma}\overline{u^1}$, integrating over $(0,1)$ and taking the real part, we get
\begin{equation}\label{edit1}
\begin{array}{ll}
\displaystyle
\frac{K}{2}\int_0^1\frac{1}{\sigma}|\la u^1|^2dx-\frac{K}{2} \int_0^1 \eta |u^1_x|^2dx+\frac{K}{2}\Re\left(\eta(1)u^1_x(1)\overline{u^1}(1)\right)-\frac{K}{2}\lim_{x\to 0} \eta u^1_x\overline{u^1}\\
\displaystyle=-\Re\left(\frac{K}{2}\int_0^1\frac{1}{\sigma}g^2\overline{u^1}dx\right)
-\Re\left(i\frac{K}{2}\int_0^1\frac{1}{\sigma} g^1\la \overline{u^1}dx\right)
\end{array}
\end{equation}
Now, multiplying the first equation in \eqref{Q56} by $\displaystyle-2\frac{x}{\sigma}\overline{u^1_x}$, integrating over $(0,1)$ and taking the real part, we get
\begin{equation}
\begin{array}{ll}
\displaystyle
\int_0^{1}\left(\frac{x}{\sigma}\right)^{\prime}|\la u^1|^2dx +\int_0^{1} \eta|u^1_x|^2dx-\int_0^{1}x\frac{b}{a}\eta|u^1_x|^2dx-\eta(1)|u^1_x(1)|^2-\frac{1}{\sigma(1)}|\la u^1(1)|^2\\
\displaystyle
+\la ^2\lim_{x\to 0} \frac{x}{\sigma}|u^1|^2=2\Re\left(\int_0^{1} g^2\frac{x}{\sigma}\overline{u^1_x}dx\right)+2\Re\left(i\la \int_0^{1} g^1\frac{x}{\sigma}\overline{u^1_x}dx\right).
\end{array}
\end{equation}
Using Lemma \ref{Lemma0} and the fact that  $\displaystyle\left(\frac{x}{\sigma}\right)^{\prime}=\ma{\frac{1}{\sigma}\left(1-x\left(\frac{a'-b}{a}\right)\right)}$, the above equation becomes
\begin{equation}\label{EE0}
\begin{array}{ll}
\displaystyle
\int_0^{1}\left(\frac{1}{\sigma}|\la u^1|^2+\eta|u^1_x|^2\right)dx=\int_0^{1}\frac{x}{\sigma}\left(\frac{a^{\prime}-b}{a}\right)|\la u^1|^2dx+\int_0^{1}x\frac{b}{a}\eta|u^1_x|^2dx+\eta(1)|u^1_x(1)|^2\\
\displaystyle+\frac{1}{\sigma(1)}|\la u^1(1)|^2
+2\Re\left(\int_0^{1} g^2\frac{x}{\sigma}\overline{u^1_x}dx\right)+2\Re\left(i\la \int_0^{1} g^1\frac{x}{\sigma}\overline{u^1_x}dx\right)
\end{array}
\end{equation}
Adding \eqref{edit1} and \eqref{EE0}, we get
\begin{equation}
\begin{array}{ll}
\displaystyle
\left(1+\frac{K}{2}\right)\int_0^{1}\frac{1}{\sigma}|\la u^1|^2+\left(1-\frac{K}{2}\right)\int_0^{1}\eta|u^1_x|^2dx=\int_0^{1}\frac{x}{\sigma}\left(\frac{a^{\prime}-b}{a}\right)|\la u^1|^2dx\\
\displaystyle+\int_0^{1}x\frac{b}{a}\eta|u^1_x|^2dx

+\eta(1)|u^1_x(1)|^2+\frac{1}{\sigma(1)}|\la u^1(1)|^2
+2\Re\left(\int_0^{1} g^2\frac{x}{\sigma}\overline{u^1_x}dx\right)+2\Re\left(i\la \int_0^{1} g^1\frac{x}{\sigma}\overline{u^1_x}dx\right)\\
\displaystyle-\frac{K}{2}\Re\left(\eta(1)u^1_x(1)\overline{u^1}(1)\right)

-\Re\left(\frac{K}{2}\int_0^1\frac{1}{\sigma}g^2\overline{u^1}dx\right)
-\Re\left(i\frac{K}{2}\int_0^1\frac{1}{\sigma} g^1\la \overline{u^1}dx\right).
\end{array}
\end{equation}
Now, we need to estimate the terms on the left hand side of the above equation. \\
$\bullet$ 
As a first step, consider the term $\displaystyle 2\Re\left(\int_0^{1} g^2\frac{x}{\sigma}\overline{u^1_x}dx\right)$: using the Cauchy-Schwarz inequality and the monotonicity of the function $\displaystyle\frac{x}{\sqrt{a}}$ in $(0,1]$, we get
\begin{equation}\label{EE1}
\left|2\Re\left(\int_0^{1} g^2\frac{x}{\sigma}\overline{u^1_x}dx\right)\right|\leq 2\Re\int_0^{1} \frac{x}{\ma{\sqrt{\sigma}}}|g^2|\frac{\sqrt{\eta}}{\sqrt{a}}|u^1_x|dx\leq m_9\left[\|\mathcal{U}\|_{\mathcal{H}_1}\|G\|_{\mathcal{H}_1}+\|G\|_{\mathcal{H}_1}^2\right],
\end{equation}
where $m_9=\frac{2}{\sqrt{a(1)}} $.\\
$\bullet$ Now, for the term $\displaystyle 2\Re\left(i\la \int_0^{1} g^1\frac{x}{\sigma}\overline{u^1_x}dx\right)$, we have
\begin{equation}\label{Q12}
\begin{array}{ll}
\displaystyle
2\Re\left(i\la \int_0^{1} g^1\frac{x}{\sigma}\overline{u^1_x}dx\right)=-2\Re\left(i\la \int_0^{1}\frac{x}{\sigma} g^1_x\overline{u^1}dx\right)-2\Re\left(i\la \int_0^{1} \left(\frac{x}{\sigma}\right)^{\prime}g^1\overline{u^1}dx\right)\\
\displaystyle+2\Re\left(i\la g^1(1)\frac{1}{\sigma(1)}\overline{u^1}(1)\right)-2\lim_{x\to 0}\Re\left(i\la g^1\frac{x}{\sigma}\overline{u^1}\right).
\end{array}
\end{equation}
Estimating the first term on the right hand side of \eqref{Q12}, we get
\begin{equation}\label{EstQ2}
\left|2\Re\left(i\la \int_0^{1}\frac{x}{\sigma} g^1_x\overline{u^1}dx\right)\right|\leq \frac{2}{\sqrt{a(1)}}\|\la u^1\|_{\frac{1}{\sigma}}\|G\|_{\mathcal{H}_1}\leq m_{10} \left[\|\mathcal{U}\|_{\mathcal{H}_1}\|G\|_{\mathcal{H}_1}+\|G\|_{\mathcal{H}_1}^2\right],
\end{equation}
with $m_{10}=\frac{2}{\sqrt{a(1)}}m_0$.
Now for the second term on the right hand side in \eqref{Q12}, using \eqref{lambdau1}, we have
\begin{equation}\label{Q13}
2\Re\left(i\la \int_0^{1} \left(\frac{x}{\sigma}\right)^{\prime}g^1\overline{u^1}dx\right)=2\Re\left(i\la \int_0^{1} \frac{1}{\sigma}g^1\overline{u^1}dx\right)+2\Re\left(i\la \int_0^{1} \frac{x}{\sigma}\left(\frac{a^{\prime}-b}{a}\right)g^1\overline{u^1}dx\right).
\end{equation}
Using \eqref{lambdau1} and the Hardy-Poincaré inequality, the last two terms in the above equation can be estimated  in the following way
\begin{equation*}
\begin{array}{ll}
\displaystyle
\left|2\Re\left(i\la \int_0^{1} \frac{1}{\sigma}g^1\overline{u^1}dx\right)\right|\leq 2 \|\la u^1\|_{\frac{1}{\sigma}}\|\frac{1}{\sqrt{\sigma}}g^1\|\leq 2m_0(\|\mathcal{U}\|_{\mathcal{H}_1}+\|G\|_{\mathcal{H}_1}) \widetilde{m}_0\|\ma{\sqrt{\eta}} g_x^1\|\\
\displaystyle
\leq m_{11} \left[\|\mathcal{U}\|_{\mathcal{H}_1}\|G\|_{\mathcal{H}_1}+\|G\|_{\mathcal{H}_1}^2\right]
\end{array}
\end{equation*}
and 
\begin{equation*}
\left|2\Re\left(i\la \int_0^{1} \frac{x}{\sigma}\left(\frac{a^{\prime}-b}{a}\right)g^1\overline{u^1}dx\right)\right|\leq 2N_1\int_0^{1} \frac{1}{\sqrt{\sigma}} |\la u^1| \frac{1}{\sqrt{\sigma}} |g^1|dx\leq m_{12}  \left[\|\mathcal{U}\|_{\mathcal{H}_1}\|G\|_{\mathcal{H}_1}+\|G\|_{\mathcal{H}_1}^2\right],
\end{equation*}
where $\displaystyle m_{11}=2m_0 \widetilde{m}_0 $ and $\displaystyle m_{12}=2m_0N_1\widetilde{m}_0$.\\
Then, thanks to the above two inequalities and \eqref{Q13}, we get
\begin{equation}\label{EstQ4}
\left|2\Re\left(i\la \int_0^{1} \left(\frac{x}{\sigma}\right)^{\prime}g^1\overline{u^1}dx\right)\right| \leq m_{13}  \left[\|\mathcal{U}\|_{\mathcal{H}_1}\|G\|_{\mathcal{H}_1}+\|G\|_{\mathcal{H}_1}^2\right],
\end{equation}
where $m_{13}=m_{11}+m_{12}$.
Now, consider the term $\displaystyle 2\Re\left(i\la g^1(1)\frac{1}{\sigma(1)}\overline{u^1}(1)\right)$. Using \ma{the conditions $g^1(1)=g^3(1)$, $u^1(1)=y^1(1)$ and the inequalities \eqref{g3L1} and \eqref{BDu}}, we get
\begin{equation}\label{EstQ3}
\left|2\Re\left(i\la g^1(1)\frac{1}{\sigma(1)}\overline{u^1}(1)\right)\right|\leq\frac{1}{\sigma(1)}\left(|g^1(1)|^2+|\la u^1(1)|^2\right) 
\leq m_{14} \left[\|\mathcal{U}\|_{\mathcal{H}_1}\|G\|_{\mathcal{H}_1}+\|G\|_{\mathcal{H}_1}^2\right],
\end{equation}
where $\ma{\displaystyle m_{14}=\frac{1}{\sigma(1)}(m_7+m_1^2)}$.\\
By  \eqref{Q12}, \eqref{EstQ2}, \eqref{EstQ4},   \eqref{EstQ3} and Lemma \ref{Lemma0}, we get
\begin{equation}\label{EE2}
\left|2\Re\left(i\la \int_0^{1} g^1\frac{x}{\sigma}\overline{u^1_x}dx\right)\right|\leq m_{15}\left[\|\mathcal{U}\|_{\mathcal{H}_1}\|G\|_{\mathcal{H}_1}+\|G\|_{\mathcal{H}_1}^2\right]
\end{equation}
where $m_{15}=m_{10}+m_{13}+m_{14}$.\\

By using Cauchy-Schwarz and Young's inequality and \eqref{BDu}, we get \
\begin{equation}\label{edit2}
\begin{array}{ll}
\displaystyle \left|\Re\left(\frac{K}{2}\int_0^1\frac{1}{\sigma}g^2\overline{u^1}dx\right)\right|\leq \frac{K}{2} \widetilde{m}_0 \left[\|\mathcal{U}\|_{\mathcal{H}_1}\|G\|_{\mathcal{H}_1}+\|G\|_{\mathcal{H}_1}^2\right],
\end{array}
\end{equation}
\begin{equation}\label{edit3}
\begin{array}{ll}
\displaystyle \left|\Re\left(i\frac{K}{2}\int_0^1\frac{1}{\sigma} g^1\la \overline{u^1}dx\right)\right| \leq \frac{K}{2} \widetilde{m}_0 m_0 \left[\|\mathcal{U}\|_{\mathcal{H}_1}\|G\|_{\mathcal{H}_1}+\|G\|_{\mathcal{H}_1}^2\right]
\end{array}
\end{equation}
and 
\begin{equation}\label{edit4}
\begin{array}{ll}
\displaystyle 
\left|\frac{K}{2}\Re\left(\eta(1)u^1_x(1)\overline{u^1}(1)\right)\right|\leq  \widetilde{m}_7 \left(1+\frac{1}{\la^2}\right)  \left[\|\mathcal{U}\|_{\mathcal{H}_1}\|G\|_{\mathcal{H}_1}+\|G\|_{\mathcal{H}_1}^2\right]
\end{array}
\end{equation}
where $\displaystyle \widetilde{m}_7=\max\left(\frac{m_8K^2}{8},\frac{m_7}{2}\right)$.
Hence, using \eqref{BDu}, \eqref{EE1},  \eqref{EE2}, \eqref{edit2}, \eqref{edit3} and \eqref{edit4} in \eqref{EE0}, we obtain
\begin{equation}\label{FinalEq}
\int_0^{1}(1+\frac{K}{2}-N_1)\frac{1}{\sigma}|\la u^1|^2dx+\int_0^{1}(1-\frac{K}{2}-N_2)\eta|u^1_x|^2dx\leq m_{16}\left(1+\frac{1}{\la^2}\right)  \left[\|\mathcal{U}\|_{\mathcal{H}_1}\|G\|_{\mathcal{H}_1}+\|G\|_{\mathcal{H}_1}^2\right],
\end{equation}
where $\displaystyle m_{16}=\max\left(m_9+m_{15}+\frac{1}{\sigma(1)}m_7+\frac{1}{\eta(1)}m_8+\frac{K}{2} \widetilde{m}_0+\frac{K}{2} \widetilde{m}_0 m_0, \widetilde{m}_7\right)$. Thus, we can conclude that \eqref{FinalEq1} and \eqref{FinalEq2} are satisfied with $m_{17}=m_{16}(1+\frac{K}{2}-N_1)^{-1}$ and $m_{18}=m_{16}(1-\frac{K}{2}-N_2)^{-1}$.\\
\underline{\textbf{Step 4.}} The aim of this step is to prove \eqref{mainresult}. \\
From \eqref{Q1},  and by using the Young's inequality,  \eqref{HP} and \eqref{FinalEq1},  we obtain
\begin{equation}\label{FinalEq3}
\begin{array}{ll}\displaystyle
\int_0^{1}\frac{1}{\sigma}| v^1|^2dx\leq 2\int_0^{1}\frac{1}{\sigma}|\la u^1|^2dx+2C_{HP}\max_{x\in[0,1]}\frac{1}{\eta}\int_0^{1}\eta| g^1_x|^2dx\\
\displaystyle
\leq 2\max\{m_{17},\widetilde{m}_0^2\}\left(1+\frac{1}{\la^2}\right) \left[\|\mathcal{U}\|_{\mathcal{H}_1}\|G\|_{\mathcal{H}_1}+\|G\|_{\mathcal{H}_1}^2\right].
\end{array}
\end{equation}
Moreover, by \eqref{est5}, \eqref{FinalEq4}, \eqref{FinalEq2},  and \eqref{FinalEq3},  we obtain
\begin{equation}\label{Editfinal}
\|\mathcal{U}\|^2_{\mathcal{H}_1}\leq m_{19} \left(1+\frac{1}{\la^2}\right)  \left[\|\mathcal{U}\|_{\mathcal{H}_1}\|G\|_{\mathcal{H}_1}+\|G\|_{\mathcal{H}_1}^2\right],
\end{equation}
where $\displaystyle m_{19}=3m_5+2(R+1)+m_{18}+2\max\{m_{17},\widetilde{m}_0^2\}$. \\
We have $\displaystyle m_{19}\left(1+\frac{1}{\la^2}\right)\|\mathcal{U}\|_{\mathcal{H}_1}\|G\|_{\mathcal{H}_1}\leq \frac{1}{4}\|\mathcal{U}\|_{\mathcal{H}_1}^2+ m_{19}^2\left(1+\frac{1}{\la^2}\right)^2\|G\|_{\mathcal{H}_1}^2$. 
Since $\displaystyle m_{19}^2\left(1+\frac{1}{\la^2}\right)^2+m_{19}\left(1+\frac{1}{\la^2}\right) \leq \left(m_{19}+\frac{1}{2}+\frac{m_{19}}{\la^2}\right)^2$, then \eqref{Editfinal} becomes
\begin{equation*}
\|\mathcal{U}\|_{\mathcal{H}_1}^2\leq \frac{4}{3}\left(m_{19}+\frac{1}{2}+\frac{m_{19}}{\la^2}\right)^2 \|G\|_{\mathcal{H}_1}^2\leq  \frac{4}{3}\left(m_{19}+\frac{1}{2}\right)^2\left(1+\frac{1}{\la^2}\right)^2 \|G\|_{\mathcal{H}_1}^2.
\end{equation*}
Thus, $\displaystyle\|\mathcal{U}\|_{\mathcal{H}_1}\leq m \left(1+\frac{1}{\la^2}\right)  \|G\|_{\mathcal{H}_1}$, with $m=\frac{1}{\sqrt{3}}(2m_{19}+1)$ and the proof of Proposition \ref{prop2'} is completed.
\hfill$\square$

\noindent \textbf{Proof of Theorem \ref{Exponential Stab2}.}
The proof of the conditions \eqref{R1} and \eqref{R2}, is similar to the proof of \eqref{H1} and \eqref{H2} respectively in Theorem \ref{Exponential Stab}.And finally, we get the conclusion applying the Huang-Pr\"{u}ss (see \cite{Huang01, pruss01}) Theorem.\hfill$\square$

\begin{rem}\label{RemarkG}
Thanks to the previous result we can improve the result given in \cite{fragnelli2022linear}. Indeed  if Hypothesis 4 in \cite{fragnelli2022linear} is satisfied then the two assumptions taken on $N_1$ and $N_2$ in Hypothesis   \ref{HypothesisExp2} satisfies $N_1<1+\frac{K}{2}$ and $N_2<1-\frac{K}{2}$ .
\end{rem}

\section{Conclusion} 
In this paper we study two systems. The first one is a degenerate wave equation in non divergence form with a drift term and localized internal degenerate damping. We prove the well-posedness of this system and the exponential stability, without any restriction on the degeneracy of the damping  coefficient (i.e $\alpha_1, \alpha_2\geq 0$).
The second system we consider consists of coupled degenerate and non-degenerate wave equations connected through transmission conditions with a  drift term and a single boundary damping at the endpoint of the non-degenerate wave equation. Also for this system we prove well-posedness and exponential stability. 

\section*{Acknowledgment}
The authors would like to thank the Project Horizon Europe Seeds {\it STEPS: STEerability and 
controllability of PDES in Agricultural and Physical models} and the Project {\it PEPS JCJC-FMHF- FR2037}.\\
Genni Fragnelli is also a member of the  {\it Gruppo Nazionale per l'Analisi Ma\-te\-matica, la Probabilit\`a e le loro Applicazioni (GNAMPA)} of the Istituto Nazionale di Alta Matematica (INdAM) and a member of {\it UMI ``Modellistica Socio-Epidemiologica (MSE)''}. She is partially supported by the GNAMPA project 2023 {\em Modelli differenziali per l'evoluzione del clima e i suoi impatti} and by FFABR {\it Fondo per il finanziamento delle attivit\`a base di ricerca} 2017.

\end{document}